\documentclass[12pt,reqno]{amsart}

\usepackage{xcolor}
\usepackage{comment}
\usepackage{amsmath}
\usepackage{amssymb}
\usepackage{amsthm}
\usepackage[latin1]{inputenc}
\usepackage{eurosym}
\usepackage{graphicx}
\usepackage{epsfig}
\usepackage{hyperref}
\usepackage{dsfont}
\usepackage[nocompress, space]{cite}

\usepackage[displaymath,mathlines]{lineno}
\allowdisplaybreaks

\usepackage{hyperref}
\usepackage{ifthen}

\newcommand{\N}{\mathbb{N}}

\newcommand{\R}{\mathbb{R}}
\newcommand{\bS}{\mathbb{S}}
\newcommand{\T}{\mathbb{T}}
\newcommand{\Z}{\mathbb{Z}}

\newcommand{\cA}{\mathcal{A}}

\newcommand{\cC}{\mathcal{C}}
\newcommand{\cE}{\mathcal{E}}

\newcommand{\cH}{\mathcal{H}}
\newcommand{\cL}{\mathcal{L}}
\newcommand{\cM}{\mathcal{M}}
\newcommand{\cP}{\mathcal{P}}
\newcommand{\cR}{\mathcal{R}}
\newcommand{\cS}{\mathcal{S}}

\newcommand{\al}{\alpha}
\newcommand{\gam}{\gamma}
\newcommand{\del}{\delta}
\newcommand{\ep}{\varepsilon}

\newcommand{\lam}{\lambda}
\newcommand{\sig}{\sigma}
\newcommand{\om}{\omega}

\newcommand{\Del}{\Delta}

\newcommand{\Lam}{\Lambda}

\newcommand{\Li}{L^{\infty}}
\newcommand{\Lip}{{\rm Lip\,}}

\newcommand{\ol}{\overline}

\newcommand{\eig}{{\rm eig}\,}

\newcommand{\proj}{{\rm proj}\,}

\newcommand{\tr}{\operatorname{tr}}

\newcommand{\supp}{\operatorname{supp}}

\newcommand{\cl}{\operatorname{cl}}

\renewcommand{\div}{\operatorname{div}}

\newcommand{\Rr}{{\mathbb{R}}}

\newcommand{\Tt}{{\mathbb{T}}}

\newcommand{\Aa}{{\mathcal{A}}}
\newcommand{\Cg}{{\mathcal{C}}}

\newcommand{\Uu}{{\mathcal{U}}}

\newcommand{\Mm}{{\mathcal{M}}}
\newcommand{\Ss}{{\mathcal{S}}}

\def\leq{\leqslant}
\def\geq{\geqslant}

\numberwithin{equation}{section}

\newtheoremstyle{thmlemcorr}{10pt}{10pt}{\itshape}{}{\bfseries}{.}{10pt}{{\thmname{#1}\thmnumber{
#2}\thmnote{ (#3)}}}
\newtheoremstyle{thmlemcorr*}{10pt}{10pt}{\itshape}{}{\bfseries}{.}\newline{{\thmname{#1}\thmnumber{
#2}\thmnote{ (#3)}}}
\newtheoremstyle{defi}{10pt}{10pt}{\itshape}{}{\bfseries}{.}{10pt}{{\thmname{#1}\thmnumber{
#2}\thmnote{ (#3)}}}
\newtheoremstyle{remexample}{10pt}{10pt}{}{}{\bfseries}{.}{10pt}{{\thmname{#1}\thmnumber{
#2}\thmnote{ (#3)}}}
\newtheoremstyle{ass}{10pt}{10pt}{}{}{\bfseries}{.}{10pt}{{\thmname{#1}\thmnumber{
A#2}\thmnote{ (#3)}}}

\theoremstyle{thmlemcorr}
\newtheorem{theorem}{Theorem}
\numberwithin{theorem}{section}

\newtheorem{proposition}[theorem]{Proposition}

\theoremstyle{thmlemcorr*}
\newtheorem{theorem*}{Theorem}
\newtheorem{lemma*}[theorem]{Lemma}
\newtheorem{corollary*}[theorem]{Corollary}
\newtheorem{proposition*}[theorem]{Proposition}
\newtheorem{problem*}[theorem]{Problem}
\newtheorem{conjecture*}[theorem]{Conjecture}

\theoremstyle{defi}

\newtheorem{hyp}{Assumption}

\newtheorem{problem}{Problem}

\theoremstyle{remexample}
\newtheorem{remark}{Remark}

\theoremstyle{plain}
\newtheorem{teo}[theorem]{Theorem}
\newtheorem{lem}[theorem]{Lemma}
\newtheorem{pro}[theorem]{Proposition}
\newtheorem{cor}[theorem]{Corollary}

\theoremstyle{ass}

\begin{document}

\title[The large time profile]{The large time profile for Hamilton--Jacobi--Bellman equations}

\author{Diogo A. Gomes}
\address[D. A. Gomes]{
	King Abdullah University of Science and Technology (KAUST), CEMSE Division, Thuwal 23955-6900, Saudi Arabia.}
\email{diogo.gomes@kaust.edu.sa}

\author{Hiroyoshi Mitake}
\address[H. Mitake]{
	Graduate School of Mathematical Sciences, 
	University of Tokyo 
	3-8-1 Komaba, Meguro-ku, Tokyo, 153-8914, Japan}
\email{mitake@ms.u-tokyo.ac.jp}

\author{Hung V. Tran}
\address[Hung V. Tran]
{
	Department of Mathematics, 
	University of Wisconsin Madison, Van Vleck hall, 480 Lincoln drive, Madison, WI 53706, USA}
\email{hung@math.wisc.edu}

\thanks{
	The work of DG was partially supported by 
	 KAUST baseline funds and 
	KAUST OSR-CRG2017-3452.
	The work of HM was partially supported by the JSPS grants: KAKENHI \#19K03580, \#19H00639, \#17KK0093, \#20H01816. 
	The work of HT was partially supported by NSF grant DMS-1664424 and NSF CAREER grant DMS-1843320.
}

\date{\today}

\begin{abstract}
Here,  we study  
the large-time limit of viscosity solutions of the Cauchy problem for second-order Hamilton--Jacobi--Bellman equations with convex Hamiltonians in the torus.
This large-time limit 
 solves the corresponding stationary problem, sometimes called the ergodic problem. 
This problem, however, has multiple viscosity solutions 
and, thus, 
a key question is  which of these solutions is selected by the limit.
%the large-time limit is not completely determined by the aforementioned result.  
Here, we provide a representation for the viscosity solution to the Cauchy problem in terms of generalized holonomic measures. Then, we use this representation to characterize the large-time limit in terms of the initial data and generalized Mather measures.
In addition, we establish various results on generalized Mather measures and duality theorems that are of independent interest.  
\end{abstract}

\maketitle

\section{Introduction}

%{\sc Paragraph 1}
%
%{\bf Establish the importance of the field, provide background facts/information, define terminology in the title if necessary.}
%
%
%{\color{red}ADD}
%
%{\color{red} Go through the paper and check that all integrals have a domain of integration.}
%

Under a broad range of conditions, as time goes to infinity, 
viscosity solutions of the Cauchy problem for Hamilton--Jacobi--Bellman equations with convex Hamiltonians 
converge to stationary solutions. 
In the last two decades,
this matter was investigated in detail. 
For first-order Hamilton--Jacobi  equations, see, for example,  \cite{FATH4,  barles2000large, MR2237158, LMT, Ichihara-Ishii} and references therein. 
For second-order  Hamilton--Jacobi--Bellman equations, 
 the large-time convergence was studied in the non-degenerate case in \cite{Barles-Souganidis2},
and in the general (possibly degenerate) case in \cite{CGMT, LN}.   
However, the stationary problem may have multiple viscosity solutions
and, thus, a natural question is which 
of these viscosity solutions is selected. 
Here, 
in Theorem \ref{thm:profile},
we answer this question by giving a characterization of the limit
in terms of generalized Mather measures.

More precisely, we study the large time behavior of the solutions of the following problem. 
\begin{problem}
	\label{P1}
Let $\T^n =\R^n/\Z^n$ be the flat $n$-dimensional torus. Consider a non-degenerate
diffusion coefficient,  $a:\T^n\to [0,\infty)$.
Fix 
a Hamiltonian $H:\T^n \times \R^n \to \R$, 
and continuous initial data, $u_0:\T^n\to\R$.
Find
a (viscosity) solution, $u:\T^n\times[0,\infty)\to\R$, to  the Hamilton--Jacobi--Bellman equation
\begin{equation} \label{eq:1}
u_t -a(x)\Delta u + H(x,Du) = 0 \quad \text{ in }  \T^n \times (0,\infty), 
\end{equation}
where 
$Du$ and $\Del u$, respectively, denote the spatial gradient and Laplacian of $u$, 
and with the initial data
\begin{equation*}
%\label{eq:1bc}
u(x,0) = u_0(x) \quad \text{ on } \T^n.
\end{equation*}
\end{problem}

Here, we {\em always} work with viscosity solutions, and, thus, the term ``viscosity" is
omitted henceforth.

Our goal  is to study the large-time behavior of the solution of Problem \ref{P1} and provide a characterization of the asymptotic profile of $u(x,t)$ as $t\to \infty$. 
For that, we
work under the following assumptions. 

\begin{hyp}
\label{A1}
The functions
$u_0$, $a$, and $H$  are of class $C^2$. 
\end{hyp}

\begin{hyp}
	\label{A2}
For every $x\in \T^n$, the map $p \mapsto H(x,p)$ is strictly convex; that is,  
$D^2_{pp} H(x,p)>0$ for all $(x,p)\in\T^n\times \Rr^n$. 
\end{hyp}
\begin{hyp}
\label{A3}	
There exist $m>1$ and $C_0>0$ such that, for all $(x,p) \in \T^n \times \R^n$,
\[
\begin{cases}
\displaystyle \frac{|p|^m}{C_0} - C_0 \leq H(x,p) \leq C_0(|p|^m+1),\\
|D_x H(x,p)| \leq C_0(1+|p|^m),\\
|D_p H(x,p)| \leq C_0(1 + |p|^{m-1}).
\end{cases}
\]	
\end{hyp}
Because of the strict convexity in Assumption \ref{A2}  and the growth conditions in Assumption \ref{A3}, the Lagrangian, 
\begin{equation}
\label{legendre}
L(x,q)=\sup_{p\in\Rr^n} \left( p \cdot q-H(x,p) \right),
\end{equation}
is  finite for all $(x,q)\in\T^n\times\R^n$. 
Moreover, by increasing $C_0$,  if necessary, we have
\[
L(x,q)\geq \frac{|q|^{m'}}{C_0}- C_0 \quad \text{ for all } (x,q) \in \T^n \times \R^n, 
\]
where 
\begin{equation}
\label{conj}
\frac 1 m + \frac 1 {m'}=1.
\end{equation}

Under Assumptions \ref{A1}-\ref{A3}, Problem \ref{P1} has a unique Lipschitz solution (see \cite{A-T, LMT} and the references therein).
Furthermore, 
as established in  \cite[Theorem 1.1]{CGMT}, if 
Assumptions  \ref{A1}-\ref{A3} hold, there exist $(u_\infty,c)\in\Lip(\T^n)\times\R$ such that  
\begin{equation}
\label{uinfdef}
	u(\cdot,t)+ct\to u_\infty\quad\text{ uniformly on} \ \T^n \ 
	\text{as} \ t\to\infty.  
\end{equation}
Moreover, $(u_\infty,c)$ solves (in the viscosity sense) the following {\em ergodic problem}.
\begin{problem}	
\label{E}
Under the setting of Problem \ref{P1}, find $(v,c)\in\Lip(\T^n)\times\R$ such that 
	\begin{equation}
	\label{ee}
	 -a(x) \Delta v+H(x,Dv)=c \quad \text{ in }  \T^n. 
	\end{equation}
\end{problem}

The constant $c$ in Problem \ref{E}, the \textit{ergodic constant},
is uniquely determined by $H$ and $a$ under Assumptions \ref{A1}-\ref{A3} 
(see \cite[Theorem 4.5, Proposition 4.8]{LMT} for instance).   
However, because the operator on the left-hand side of \eqref{ee} is not strictly
monotone in $v$,
Problem \ref{E} may have multiple solutions as the examples in \cite[Chapter 6]{LMT} illustrate. 
Thus, a key question is which 
solution is selected by the large-time limit.
Here, we answer this question and give a representation formula for the limit function $u_\infty$ that highlights the  dependence on
the initial data $u_0$. 

\begin{remark}
\label{czero}
By adding a constant to $H$, we can assume that 
the ergodic constant in Problem \ref{E} vanishes, that is,  $c=0$. Therefore, we \textit{always} assume that $c  = 0$ henceforth. 

%\smallskip
%
%We also emphasize that the strict convexity of $H$ in Assumption \ref{A2} is essential to obtain the convergence result \eqref{uinfdef}.
%In other words, mere convexity of $H$ is not enough to guarantee \eqref{uinfdef}.
%Indeed, it is clear that $u(x,t)=\sin(2\pi(x+t))$ is a solution to 
%\[
%u_t+|u_x-1| -1 =0 \quad \text{ in } \T \times (0,\infty),
%\]
% and $u(\cdot,t)$ does not converge as $t \to \infty$.
%This example was pointed out in \cite{barles2000large}. 
\end{remark}

The selection problem, that is, to characterize which of the solutions 
of  Problem \ref{E} is selected by the limit $t\to \infty$, was first addressed
for first-order Hamilton--Jacobi equations in 
 \cite[Theorem 3.1]{MR2237158}. Moreover, the last two authors of the present paper gave an elementary proof of that selection result in \cite[Theorem 1.3]{MT2}.
 These results characterize $u_\infty$ in the first-order case. 
 If $a > 0$ on $\Tt^n$,   Problem \ref{E} has a unique solution (up to additive constants), 
 see, for example, \cite{Barles-Souganidis2, G}. 
 Thus, in this case, the selection problem is simple since
 only the additive constant needs to be determined. 
 However, prior to this work, 
 there was no characterization of the limit for the general case where $a$ can degenerate. 
 In this case, uniqueness for Problem \ref{E}  may not hold and
 the techniques for first-order equations that use deterministic trajectories do not apply.
The present paper closes a key gap in the understanding of
 the large-time limit of solutions of Problem \ref{P1} by establishing a representation result for $u_\infty$ in terms of 
 generalized Mather measures and the initial data.

\subsection{Main results on representation formulas and  limiting profiles}
 
A holonomic measure on $\Tt^n\times \Rr^n$ %for first-order Hamilton--Jacobi equations 
is a probability measure, $\mu\in\cP(\Tt^n\times \Rr^n)$,  
that satisfies
\[
\int_{\Tt^n\times \Rr^n} q\cdot D\varphi(z) \,d\mu(z, q)=0 
\]
for any $\varphi\in C^1(\Tt^n)$. 
In Aubry-Mather theory, the set of holonomic measures is of paramount importance
as these measures generalize closed curves.
These measures play a critical role in understanding Lagrangian dynamics and viscosity
solutions of first-order Hamilton-Jacobi equations.

Generalized Mather measures
were introduced in the setting of the stochastic Mather problem in \cite{G} 
and then further studied for fully nonlinear elliptic equations in \cite{G10}. 
Later, these measures were studied in different contexts and used to study asymptotic problems 
in \cite{CDG,MR2458239, MR3581314, MT2} and \cite{MR3670619}. 

%\smallskip
Here, we define a class of generalized holonomic
measures that are suitable for the study of time-dependent problems. 
Let $m$ be the exponent in Assumption \ref{A3} and $m'$ its conjugate exponent as in \eqref{conj}.
Fix $\zeta$ with $1<\zeta<m'$.
Let $t_0, t_1\in \Rr$ with $0\leq t_0<t_1$. 
Let $\cR(\Tt^n\times \Rr^n\times [t_0,t_1])$  be  the set of all Radon measures on $\Tt^n\times \Rr^n\times [t_0,t_1]$, and 
$\cR^+(\Tt^n\times \Rr^n\times [t_0,t_1])$  be  the set of all non-negative Radon measures on $\Tt^n\times \Rr^n\times [t_0,t_1]$.
Let $\cP(\Tt^n)$ and $\cP(\T^n\times \R^n)$ be the set of probability measures on $\Tt^n$ and $\T^n \times \R^n$, respectively.

For $\nu_0, \nu_1\in\cP(\T^n)$, $\cH(\nu_0,\nu_1;t_0,t_1)$ is the set of all 
	$\gamma\in \cR^+(\Tt^n\times \Rr^n\times [t_0,t_1])$
	 satisfying
\begin{equation*}
%\label{zetaint}
\int_{\T^n \times \R^n\times[t_0,t_1]} |q|^\zeta \,d\gam(z,q,s)<\infty
\end{equation*}
and	 	 
	\begin{align*}
	&\int_{\T^n \times \R^n\times[t_0,t_1]} 
	\left(\varphi_t(z,s)-a(z)\Delta \varphi(z,s)+q\cdot D\varphi(z,s) \right)
	\,d\gam(z,q,s) \nonumber \\
	=\, &\, 
	\int_{\T^n}\varphi(z,t_1)\,d\nu_1(z)
	-\int_{\T^n}\varphi(z,t_0)\,d\nu_0(z)  
%	\label{def:1}
	\end{align*}
	for all $\varphi\in C^2(\T^n\times[t_0,t_1])$. 
	Moreover, we set 
	\[
	\cH(\nu_1;t_0,t_1):=\bigcup_{\nu_0\in\cP(\T^n)} \cH (\nu_0,\nu_1;t_0,t_1). 
	\]
Fix any $\nu_1 \in\cP(\T^n)$. 
It is worth emphasizing that while we do not know that $\cH(\nu_0,\nu_1;t_0,t_1)\neq \emptyset$ for each $\nu_0\in\cP(\T^n)$, 
the set $\cH(\nu_1;t_0,t_1)$ is always non-empty as shown in  Corollary \ref{cor:non-empty}. 
For  $\gamma\in \cH(\nu_1;t_0,t_1)$, we denote by $\nu^{\gamma}$ the unique element in $\cP(\T^n)$ such that $\gamma\in \cH(\nu^{\gamma},\nu_1;t_0,t_1)$. 

In Section \ref{repsec}, we use the measures in  $\cH(\nu_1;t_0,t_1)$ to obtain the representation formula for the solutions of Problem \ref{P1}. This formula generalizes the optimal control representation for first-order Hamilton--Jacobi equations and the stochastic optimal control representation for second-order Hamilton--Jacobi equations.

\begin{teo}\label{thm:rep}
	Let $u$ solve Problem \ref{P1}.  Suppose that Assumptions  \ref{A1}-\ref{A3} hold.
	Then,
	for any $\nu\in\cP(\T^n)$ and $t>0$, we have
	\begin{equation*}
	%\label{eq:representation}
	\int_{\T^n}u(z,t)\,d\nu(z)=
	\inf_{\gam\in\cH(\nu;0,t)}
	\left[
	\int_{\T^n\times\R^n\times[0,t]}L(z,q)\,d\gam(z,q,s)+\int_{\T^n}u_0(z)\,d\nu^\gamma(z)
	\right].
	\end{equation*}
\end{teo}

To understand the possible limits as $t\to\infty$ of the solutions of Problem \ref{P1}, 
we rely on a characterization of the solutions of Problem \ref{E} established in 
\cite{MT2}. This characterization uses generalized Mather measures. 
As in the time-dependent case, we begin by defining the set of stationary generalized holonomic measures.
Let
\begin{multline*}
\cH:=\Big\{\mu\in\cP(\T^n\times\R^n) \,:\,
\int_{\T^n \times \R^n} |q|^\zeta \,d\mu(z,q)<\infty,\\
\int_{\T^n\times\R^n}\left(-a(z)\Delta \varphi(z)+q\cdot D\varphi(z)\right)\,d\mu(z,q)=0 
\quad\text{ for all} \ \varphi\in C^2(\T^n)\Big\} . 
\end{multline*}
We consider the variational problem 
\begin{equation}\label{def:Mather1}
\inf_{\mu\in\cH}
\int_{\Tt^n\times \Rr^n}L(x,q)\,d\mu(x,q). 
\end{equation}
A {\em generalized Mather measure} is a solution of the minimization problem  in \eqref{def:Mather1}. We denote by
 $\widetilde{\cM}$ the set of all generalized Mather measures.  
Moreover, the set of all generalized projected Mather measures, 
$\cM$, is the set of all measures in  $\T^n$ 
which are the projections to $\T^n$ of generalized Mather measures; that is, 
$\nu \in \cM$ if there exists $\mu \in \widetilde{\cM}$ such that
for all $\varphi\in C(\Tt^n)$, 
\[
\int_{\Tt^n } \varphi(x)d \nu(x)=\int_{\Tt^n\times  \Rr^n } \varphi(x)d \mu(x,q).
\]

\begin{remark}
%\label{Rzero}	
Since we normalize that the ergodic constant to be $0$ as noted in Remark \ref{czero}, 
we have 
\begin{equation}\label{eq:erg-const}
\inf_{\mu\in\cH} \int_{\Tt^n\times \Rr^n} L(x,q)\,d\mu(x,q)=0,  
\end{equation}
%See, for example, \cite[Lemma 6.12]{LMT} for a proof. 
as shown in  \cite[Propositions 2.2 and 2.3]{MR3581314}.
See also Proposition \ref{prop:erg} in Appendix \ref{Appen:B} for the general diffusion matrix case. 
\end{remark}

Before stating our main results, we recall a uniqueness result for Problem \ref{E}. 
Let  $v_1$ and $v_2$ be two solutions to Problem \ref{E} with
	\[
	\int_{\T^n}v_1(x)\, d\nu=\int_{\T^n}v_2(x)\, d\nu
	\quad
	\text{ for all} \ \nu\in\cM. 
	\] 
Then, according to \cite[Theorem 1.1]{MT2}, we have 
\[
v_1= v_2 \qquad\text{on} \ \T^n.
\]
Therefore,
to characterize $u_\infty$ uniquely, 
 it  is enough to specify the integral of $u_\infty$ with respect to generalized projected Mather measures. 
This  characterization is given  in Theorem \ref{thm:profile} below. The corresponding
proof is presented 
in Section \ref{PBAS}.

\begin{teo}\label{thm:profile}
	Suppose that Assumptions  \ref{A1}-\ref{A3} hold.
	Let $u$ solve Problem \ref{P1} and
	$u_\infty$ be the large time limit of $u$ as in \eqref{uinfdef}.  
	Then, for any $\nu\in\cM$, we have
	\begin{equation}
	\label{maver}
	\int_{\T^n} u_\infty(z)
	 \,d\nu(z) = \inf_{\nu_0\in\cP(\T^n)} 
	\left[d(\nu_0,\nu) + \int_{\T^n} u_0(z)\,d\nu_0(z)  \right], 
	\end{equation}
where for $\nu_0, \nu_1 \in \cP(\T^n)$, $d(\nu_0,\nu_1)$ is 
the  {\em generalized Ma\~n\'e critical potential} between $\nu_0$ to $\nu_1$ given by
\begin{align}
d(\nu_0,\nu_1):=
\inf_{\substack{\gam\in \cH(\nu_0,\nu_1;0,t)\\ t>0}}  
\int_{\T^n \times \R^n \times [0,t]} L(x,q) \, d\gam(x,q,s).
\label{def:d}
\end{align}
\end{teo}
%
%
%Here, $\cH(\nu_0,\nu_1;0,t)$ is a family of nonnegative Radon measures defined 
%on $\T^n\times\R^n\times[0,t]$, which is defined by Definition {\rm\ref{def:H}}.  
%

We also notice that Theorem \ref{thm:profile} answers the open question addressed in \cite[Section 5.7, (iv)]{LMT}. 
We refer to \cite{CGMT, MR3581314, MT2, LMT} for earlier developments. 

\begin{remark}
The assumption that the initial data, $u_0$,  is $C^2$ is not critical as by approximation, we can solve Problem \ref{P1} with continuous initial data, and prove the large time convergence in
 \eqref{uinfdef}. 
The regularity of $a$ and $H$ is more delicate as several estimates require quantitative bounds for their derivatives. 
%These bounds are detailed in the next two assumptions. 
\end{remark}

\begin{remark}\label{rem:general case}
We choose to consider \eqref{eq:1} with a specific diffusion matrix $A(x)=a(x)I_n$ for $x\in \T^n$, 
where $I_n$ in the identity matrix of size $n$, to simplify the presentation. However,
the results in Theorems \ref{thm:rep}--\ref{thm:profile} can be generalized
for more general diffusion matrices. The precise statements and  proofs are 
more delicate and are discussed in 
Theorems \ref{thm:rep B} and \ref{thm:profileB} in Appendix \ref{Appen:B}. 
\end{remark}

\subsection{Duality results}
For $t \ge 0$, let $S_t:C(\T^n)\to C(\T^n)$  be the solution operator for Problem \ref{P1}; that is, 
$S_t u_0(\cdot)=u(\cdot, t)$, where $u$ solves Problem \ref{P1} with initial data $u_0$. 
We assume the following compactness property for $S_t$.

\begin{hyp}
\label{A4} 
Let $X_{0}:=\{u_0\in C(\Tt^n) \,:\, \min_{\T^n}u_0=0 \}$. 
For any $t>0$,  $S_t(X_{0})$ is compact in $C(\Tt^n)$.

\end{hyp}

The map $S_t$ often enjoys the compactness
in the preceding assumption. 
Two important cases are examined in 
Theorem \ref{thm:A4} below. 
Under Assumption \ref{A4}, we have the following remarkable 
uniform convergence property that is proved at the end of Section 
\ref{PBAS}. 

\begin{proposition}
	\label{conv}
	Consider the setting of Problem \ref{P1} and
	suppose that Assumptions  \ref{A1}-\ref{A4} hold. 
	Given $\ep>0$, there exists $T>0$ such that
	for any initial data $u_0\in C^2(\Tt^n)$, 
	the corresponding solution, $u$, to Problem \ref{P1}
	satisfies
	\[
	\|u(\cdot, t) -u_\infty\|_{L^\infty(\T^n)}<\ep
	\]
	for all $t>T$, 
	where
	\[
	u_\infty(x)=\lim_{t\to \infty} u(x, t) \quad\text{ for } x\in \Tt^n. 
	\]
\end{proposition}

The study of Hamilton--Jacobi--Bellman equations using duality theory can be traced
back to \cite{LewisVinter, FlemingVermes1, FlemingVermes2}.
In the context of Aubry-Mather theory, the elliptic case 
was considered in \cite{G}, the fully nonlinear case in 
\cite{G10}, and the parabolic case in \cite{MR2883292}. 
In a sequence of works \cite{Mikami, Mikami2, Mikami-T}, 
various duality theorems associated with Hamilton--Jacobi--Bellman equations \eqref{eq:1}
were studied in the  non-degenerate case.  The connection between Aubry-Mather theory and
optimal transport was explored in \cite{Bernard-Buffoni}.
We also refer to \cite{DGKP} for recent work concerning stochastic optimal transport.

For $t>0$, let $h_t:\cP(\T^n)\times\cP(\T^n)\to\R\cup\{+\infty\}$ be 
\begin{align*}
h_t(\nu_0,\nu_1)=
\left\{
\begin{array}{ll}
\displaystyle
\inf_{\gam\in\cH(\nu_0,\nu_1;0,t)}
\int_{\T^n\times\R^n\times[0,t]}L(x,q)\,d\gam 
& 
\quad \text{ if} \ \cH(\nu_0,\nu_1;0,t)\not=\emptyset, \\ 
+\infty & 
\quad \text{ if} \ \cH(\nu_0,\nu_1;0,t)=\emptyset
\end{array}
\right.  
%\label{func:ht}
\end{align*}
for $\nu_0, \nu_1\in \cP(\Tt^n)$.
In Section \ref{rdualty}, we apply the Legendre-Fenchel-Rockafellar theorem  
to obtain representation formulas for $h_t$ and $d$.  

\begin{teo}\label{thm:duality-h}
	Consider the setting of Problem \ref{P1} and
	suppose that Assumptions  \ref{A1}-\ref{A3} hold.	
	Suppose that $h_t(\nu_0,\nu_1)<+\infty$ for  given $t>0$ and $\nu_0, \nu_1\in\cP(\T^n)$. 
	Then,
	\begin{equation}\label{eq:dual-ht}
	h_t(\nu_0,\nu_1)
	=\sup_{w\in\cE} \left[
	\int_{\T^n}w(z,t)\,d\nu_1(z)
	-
	\int_{\T^n}w(z,0)\,d\nu_0(z)\right],
	\end{equation}
where
\[
\cE:=\{w\in\Lip(\T^n\times[0,t]) \,:\, w \ \text{is a viscosity solution to \eqref{eq:1}}\}. 
\]
\end{teo}

The preceding result extends the duality formula derived in \cite{OttoVillani}
for the representation of the $2$-Wasserstein distance using the Benamou-Brenier dynamic formulation 
of optimal transport \cite{BB}.
 In the non-degenerate setting, a version of 
 Theorem \ref{thm:duality-h} was proven in \cite{Mikami, Mikami2, Mikami-T}.
For first-order equations, see \cite{Bernard-Buffoni}. 
The degenerate setting addressed here is substantially more complex, 
as the proof of Theorem \ref{thm:duality-h} requires a
delicate approximation argument.

Under the compactness assumption, Assumption \ref{A4}, we  obtain the following duality result for $d$ restricted to the projected Mather set, $\Mm$. 
\begin{teo}\label{thm:duality-d}
Consider the setting of Problem \ref{P1} and
suppose that Assumptions  \ref{A1}-\ref{A4} hold. 	
Let $\nu_0, \nu_1\in \Mm$ and suppose that 
$d(\nu_0,\nu_1)<\infty$.
Then,  
\[
d(\nu_0,\nu_1)= \sup_{w\in\cS}\left[
\int_{\T^n}w(z) \,d\nu_1(z)
-
\int_{\T^n}w(z) \,d\nu_0(z)\right],
\]	
where
\begin{equation}\label{def:Ss}
\cS:=\{w\in\Lip(\T^n) \,:\, w \ \text{is a viscosity solution to \eqref{ee}}\}. 
\end{equation}

%
%	Then, 
%	\begin{equation}\label{eq:dual-d}
%	d(\nu_0,\nu_1)
%	=\sup\left[
%	\int_{\T^d}w\,d\nu_1
%	-
%	\int_{\T^d}w\,d\nu_0\right], 
%	\end{equation}
%	where the supremum is taken over all viscosity solutions $w$ of 
%Problem \ref{E}. 
\end{teo}

Finally, we discuss the compactness assumption, Assumption \ref{A4}.

\begin{teo}\label{thm:A4}
Consider the setting of Problem \ref{P1} and
suppose that Assumptions  \ref{A1}-\ref{A3} hold. 
Then Assumption \ref{A4} holds if $a \equiv 0$ or $m>2$.
\end{teo}

\medskip
The paper is organized as follows. Section \ref{repsec} is devoted to the proof of Theorem \ref{thm:rep}. In Section \ref{S3}, we discuss the properties
of generalized holonomic measures.
Next, in Section \ref{PBAS}, 
we introduce 
the generalized Peierls barrier and Aubry set and  prove Theorem \ref{thm:profile}. 
Section \ref{rdualty} is devoted to the proof of Theorems \ref{thm:duality-h} and \ref{thm:duality-d}. 
In Section \ref{sec:A4}, we discuss Assumption \ref{A4}, and give the proof of Theorem \ref{thm:A4}. 
In Appendix \ref{Appen:A}, we discuss the approximation of viscosity solutions of \eqref{eq:1} by $C^2$-subsolutions of an approximate equation. 
Finally, in Appendix \ref{Appen:B}, we consider the Hamilton--Jacobi--Bellman equation with a general diffusion matrix, and generalize Theorems \ref{thm:rep} and \ref{thm:profile} for this case.

\bigskip
\noindent {\bf Acknowledgements.} 
We would like to thank Hitoshi Ishii for his suggestions on the approximations of viscosity solutions and subsolutions in Appendix \ref{Appen:B}.
We are grateful to Toshio Mikami for the discussions on Theorem \ref{thm:rep} and for giving us relevant references on the duality result in Theorem \ref{thm:duality-h}.

\section{Representation formula for solutions}
\label{repsec}

In this section, we prove Theorem \ref{thm:rep}, which gives a 
variational representation formula for solutions to Problem \ref{P1} 
 in terms of generalized holonomic measures.
We begin by establishing a lower bound for this variational formula.

\begin{lem}
\label{VL}
Suppose that Assumptions \ref{A1}-\ref{A3} hold.
Let $u$ solve Problem \ref{P1}. Then, for all $\nu_1 \in \cP(\T^n)$ and $t>0$,
\begin{equation*}%\label{ineq:VL}
\int_{\Tt^n}u(z,t)\,d\nu_1(z) \leq
\inf_{\gamma\in \cH(\nu_1; 0,t)}
 \left[
 \int_{\T^n \times \R^n\times[0,t]} 
 L(z,q)\,d\gam(z,q,s)+
\int_{\T^n}u(z,0)\,d\nu^\gamma(z)
\right]. 
\end{equation*}
\end{lem}
\begin{proof}
Let $u$ solve Problem \ref{P1}. Because Assumptions \ref{A1}-\ref{A3} hold, 
$u$ is globally Lipschitz continuous on $\T^n\times[0,\infty)$ 
(see \cite[Proposition 3.5]{A-T}, \cite[Proposition 4.15]{LMT} for instance).

Fix smooth symmetric standard mollifiers, 
$\theta \in C_c^\infty(\R^{n},[0,\infty))$ and $\rho \in C_c^\infty(\R,[0,\infty))$. 
More precisely, $\supp \theta\subset\ol{B}(0,1)\subset\R^{n}$, $\supp \rho\subset\ol{B}(0,1)\subset\R$, $\theta(x)=\theta(-x)$, $\rho(s)=\rho(-s)$, and $\|\theta\|_{L^{1}(\R^{n})}=\|\rho\|_{L^{1}(\R)}=1$. 
For $\alpha>0$,  let 
$\theta^\alpha(x)=\alpha^{-n} \theta(\alpha^{-1}x)$ for $x\in \R^{n}$, 
and $\rho^{\alpha}(t)=\alpha^{-1} \rho(\alpha^{-1}t)$ for $t\in \R$. 
Let $u^{\al}$ be the function given by \eqref{func:w-al}; that is,
\begin{equation*}
u^{\al}(x,t)=
\int_{0}^{\infty} \rho^\al(s)\int_{\T^n} \theta^\al(y)u(x-y,t-s)\,dy ds \quad \text{ for $(x,t)\in\T^n\times[\al,\infty)$.}
\end{equation*}
Set 
\[
\tilde{u}(x,t):=u^{\al}(x,t+\alpha) \quad \text{ for all $(x,t)\in\T^n\times[0,\infty)$. }
\]
We notice $\tilde u$ is a $C^2(\T^n\times[0,\infty))$ subsolution to \eqref{eq:alpha}; that is,

\begin{equation*}
\tilde{u}_t - a(x)\Del \tilde{u} + H(x,D\tilde{u}) \leq C \al^{1/2} \quad \text{ on } \T^n \times [0, \infty) 
\end{equation*}
as stated in Proposition \ref{prop:approx A} in Appendix \ref{Appen:A}.

Fix any $\gamma\in \cH(\nu_1; 0,t)$. 
By the definition of $\cH(\nu_1; 0,t)$, 
\begin{align*}
&\int_{\T^n \times \R^n\times[0,t]} 
\left( \tilde u_t(z,s)-a(z) \Del \tilde u(z,s)+q\cdot D \tilde u(z,s) \right)
\,d\gam (z,q,s) \nonumber \\
=\, & 
\int_{\Tt^n}\tilde u(z,t)\,d\nu_1 (z)
-\int_{\T^n}\tilde u(z,0)\,d\nu^{\gamma}(z).  
\end{align*}
Because of the Legendre transform definition in \eqref{legendre}, 
\[
q\cdot D\tilde u(z,s)\leq L(z,q)+H(z, D\tilde u(z,s)). 
\]
Accordingly, we have 
\begin{align*}
&
\int_{\T^n}\tilde{u}(z,t)\,d\nu_1 -\int_{\T^n}\tilde{u}(z,0)\,d\nu^{\gamma}(z)\\
=\,&\, 
\int_{\T^n\times\R^n\times[0,t]} \left(\tilde{u}_t-a(z)\Del \tilde u+q\cdot D\tilde{u} \right)\,d\gam (z,q,s)\\
\le\,&\,  
\int_{\T^n\times\R^n\times[0,t]}\left(\tilde{u}_t-a(z)\Del \tilde u+H(z,D\tilde{u})+L(z,q) \right)\,d\gam (z,q,s)\\
\le\,&\,  
\int_{\T^n\times\R^n\times[0,t]}L(z,q)\,d\gam (z,q,s)
+C \al^{1/2} t. 
\end{align*} 
Sending $\al \to 0$, we obtain
\[
\int_{\Tt^n}u(x,t)\,d\nu_1 \leq \int_{\T^n \times \R^n\times[0,t]}  L(z,q)\,d\gam(z,q,s)+
\int_{\T^n}u(z,0)\,d\nu^\gamma(z).
\]
Because $\gamma\in \cH(\nu_1; 0,t)$ is arbitrary, the statement follows. 
\end{proof}

\begin{remark}
If $a>0$ on $\T^n$, that is \eqref{eq:1} is uniformly parabolic, the solution to Problem \ref{P1}, $u$, is in $C^2(\T^n\times[0,\infty))$. 
In this case, in the preceding proof, we do not need 
to use 
 $u^\al$, the smoothed version of $u$. 
However, if $a$ degenerates, we expect only Lipschitz regularity for $u$. 
Therefore, we need to build  $C^2$  approximate subsolutions. This construction is rather technical  and is described in Proposition \ref{prop:approx A} in Appendix \ref{Appen:A}. The general diffusion case, 
is examined in 
 Proposition \ref{prop:approx} in Appendix \ref{Appen:B}.
\end{remark}

To prove the opposite bound, we consider the following regularized version of Problem \ref{P1}.
\begin{problem}
	\label{P2} In the setting of Problem \ref{P1} and
	for $\ep>0$,  
	find $u^\ep:\Tt^n\times [0,\infty)\to \Rr$ solving
	\begin{equation} \label{eq:ap}
	\begin{cases}
	u^\ep_t -a(x) \Delta u^\ep+ H(x,Du^\ep) = \ep \Del u^\ep \quad &\text { in } \T^n \times (0,\infty),\\
	u(x,0) = u_0(x) \quad &\text{ on } \T^n.
	\end{cases}
	\end{equation}
\end{problem}

If Assumptions \ref{A1}-\ref{A3} hold, Problem \ref{P2} has a unique solution, $u^\ep\in C^2(\Tt^n\times [0,+\infty))$ due to the added viscous term.   
Moreover, $u^\ep$ is Lipschitz continuous uniformly in $\ep \in (0,1)$.
Further, by standard viscosity solution theory, $u^\ep\to u$ locally uniformly on $\T^n \times [0,\infty)$, as $\ep\to 0$, where $u$ solves Problem \ref{P1}.

Now, we use  the nonlinear adjoint method, see
\cite{E3, T1}, to construct measures that satisfy 
an approximated holonomy condition.

\begin{lem}
	\label{L1}
	For $\ep>0$,  
	let $u^\ep$ solve Problem \ref{P2}. 
	Then, for any $\nu_1\in \cP(\Tt^n)$, there exist a Radon measure $\gamma^\ep\in\cR(\Tt^n\times (t_0,t_1))$ and a probability measure $\nu_0^{\ep}\in \cP(\Tt^n)$ such that  
\begin{align}
\notag
&\int_{\T^n \times \R^n\times[t_0,t_1]} 
\left(\varphi_t(z,s)-a(z) \Delta  \varphi(z)-\ep \Delta \varphi(z)+q\cdot D\varphi(z,s) \right)
\,d\gam^\ep(z,q,s) \nonumber \\
=&\, 
\int_{\T^n}\varphi(z,t_1)\,d\nu_1(z)
-\int_{\T^n}\varphi(z,t_0)\,d\nu^\ep_0(z)  
\label{shp}
\end{align}	
for all $\varphi\in C^2(\Tt^n\times [t_0,t_1])$.
Moreover, $\gamma^\ep$-almost everywhere, 
$q=D_pH(x, Du^\ep)$.
\end{lem}

\begin{proof}
For $\varphi\in C^2(\Tt^n\times [t_0,t_1])$, the linearization of \eqref{eq:ap} around  $u^\ep$ is
	\[
	\cL^\ep[\varphi]= 
	\varphi_t + D_pH(x,Du^\ep)\cdot D\varphi - (a(x)+\ep)\Del \varphi.
	\]
	Accordingly, the corresponding adjoint equation is the Fokker-Planck equation
	\begin{equation} \label{eq:adjoint}
	\begin{cases}
	-\sig^\ep_t -\Del (a(x)\sig^\ep) - \text{div}(D_pH(x,Du^{\ep})\sig^{\ep}) = \ep \Del \sig^\ep \quad &\text { in } \T^n \times (t_0,t_1),\\
	\sig^\ep(x,t_1) = \nu_1 \quad &\text{ on } \T^n.    
	\end{cases}
	\end{equation}
	By standard properties of the Fokker-Planck equation,
	\[
	\sig^{\ep}>0 \ \text{on} \ \T^n\times[t_0,t_1) \quad \text{ and } \quad 
	\int_{\T^n} \sig^{\ep}(x,t)\,dx=1 \ \text{for all} \ t\in[t_0,t_1].
	\]
	Next, 
	for each $\ep>0$ and $t\in[t_0,t_1]$, 
	let $\beta^{\ep}_t \in \mathcal{P}(\T^n \times \R^n)$ 
	be the probability measure determined by
	\begin{equation*}%\label{def-nu-ep-second}
	\int_{\T^n} \psi(x,Du^{\ep}) \sig^{\ep}(x,t)\,dx
	=\int_{\T^n \times \R^n} \psi(x,p)\,d\beta^{\ep}_t(x,p)
	\end{equation*}
	for all $\psi \in C_c(\T^n \times \R^n)$.
	For $t\in[t_0,t_1]$, let $\gam^{\ep}_t \in \mathcal{P}(\T^n \times \R^n)$ be the pullback of $\beta^{\ep}_t$ by the map $\Phi(x,q)=(x,D_q L(x,q))$, that is,  
	\begin{equation*}%\label{def-mu-ep-second}
	\int_{\T^n \times \R^n} \psi(x,p)\,d\beta^{\ep}_t(x,p)=\int_{\T^n \times \R^n} \psi(x,D_q L(x,q))\,d\gam^{\ep}_t(x,q)
	\end{equation*}
	for all $\psi \in C_c(\T^n \times \R^n)$.

	Define the measures $\beta^\ep, \gam^\ep \in \cR(\T^n \times \R^n\times[t_0,t_1])$
	by  
	\begin{equation*}	
	\int_{\T^n\times\R^n\times[t_0,t_1]}f\,d\beta^\ep
	=
	\int_{t_0}^{t_1}\int_{\T^n\times\R^n}f(\cdot,t)\,d\beta^\ep_t\,dt, 
	\end{equation*}
	and
		\begin{equation*}
	\int_{\T^n\times\R^n\times[t_0,t_1]}f\,d\gam^\ep
	=
	\int_{t_0}^{t_1}\int_{\T^n\times\R^n}f(\cdot,t)\,d\gam^\ep_t\,dt 
	\end{equation*}
	for any $f\in C_c(\T^n\times\R^n\times[t_0,t_1])$.
	
	Multiplying the first equation in \eqref{eq:adjoint} by  an arbitrary function, $\varphi\in C^2(\T^n\times[t_0,t_1])$, and integrating on $\T^n$, we gather
	\begin{align*}
	\ep\int_{\T^n}\sig^\ep\Del\varphi\,dx
	=&\, 
	-\int_{\T^n}\varphi\sig^\ep_t\,dx
	+\int_{\T^n}(-a(x)\Del \varphi +D_pH(x,Du^\ep)\cdot D\varphi)\sig^\ep\,dx\\ 
	=&\, 
	-\int_{\T^n}(\varphi\sig^\ep)_t\,dx
	+\int_{\T^n}(\varphi_t-a(x)\Del \varphi +D_pH(x,Du^\ep)\cdot D\varphi) \sig^\ep\,dx. 
	\end{align*}
	Next, integrating on $[t_0,t_1]$, we deduce the identity
	\begin{align*}
	&\ep\int_{t_0}^{t_1}\int_{\T^n}\sig^\ep\Del\varphi\,dxdt
	=\ep\int_{\T^n\times\R^n\times[t_0,t_1]}\Del\varphi\,d\gamma^{\ep}(x,q,t) 
	\\
	=&\, 
	-\int_{t_0}^{t_1}\int_{\T^n}(\varphi\sig^\ep)_t\,dxdt
	+\int_{t_0}^{t_1}\int_{\T^n}(\varphi_t-a(x)\Del \varphi +D_pH(x,Du^\ep)\cdot D\varphi) \sig^\ep\,dxdt\\
	=&\, 
	-\left[\int_{\T^n}\varphi(\cdot,t_1)\sig^\ep(\cdot,t_1)\,dx-\int_{\T^n}\varphi(\cdot,t_0)\sig^\ep(\cdot,t_0)\,dx\right]
	\\
	&
	\quad \quad\quad\quad\quad \quad\quad\quad+\int_{t_0}^{t_1}\int_{\T^n\times\R^n}(\varphi_t-a(x)\Del \varphi +q\cdot D\varphi) \,d\gam^\ep_t(x,q)\\
	=&\, 
	-\left[\int_{\T^n}\varphi(\cdot,t_1)\,d\nu_1-\int_{\T^n}\varphi(\cdot,t_0)\,d\nu_0^\ep\right]\\
	&
	\quad \quad\quad\quad \quad \quad\quad\quad+\int_{\T^n\times\R^n\times[t_0,t_1]}(\varphi_t-a(x)\Del \varphi +q\cdot D\varphi)\, d\gam^\ep(x,q,t),  
	\end{align*}
	where $d\nu_0^{\ep}:=\sigma^\ep(x,t_0)\,dx$, which implies the desired result. 
\end{proof}

\begin{cor}\label{cor:non-empty}
Under Assumptions \ref{A1}--\ref{A3}, for all $0<t_0<t_1$ and $\nu_1\in\cP(\T^n)$, 
\begin{equation}
\label{nemp}
\cH(\nu_1;t_0,t_1)\neq \emptyset. 
\end{equation}
\end{cor}
\begin{proof}
Let $\nu_0^{\ep}\in \cP(\Tt^n)$ and $\gamma^\ep\in\cR(\T^n \times \R^n \times [t_0,t_1])$ be 
the measures 
given by Lemma \ref{L1}. 
Because $u^\ep$ is Lipschitz in $x$, there exists $C>0$
such that 
\[
\|Du^\ep(\cdot,t)\|_{\Li(\T^n)}\le C
\]
for all $t\in[t_0,t_1]$. Therefore, there exists a sequence  
$\{\ep_j\}\to 0$, 
$\nu_0\in \cP(\Tt^n)$,  and $\gamma \in \cR(\T^n\times\R^n\times[t_0,t_1])$ such that 
\begin{equation}\label{subseq1}
\nu_0^{\ep_j}\rightharpoonup \nu_0 \quad \text{and} \quad
\gam^{\ep_j}\rightharpoonup \gam
\quad\text{as} \ \ j\to\infty,  
\end{equation}
weakly in $\cP(\Tt^n)$ and $\cR(\T^n\times\R^n\times[t_0,t_1])$, respectively. 
Thus, 
using \eqref{shp}, we conclude that 
$\gam\in\cH(\nu_0,\nu_1;t_0,t_1)$, which implies \eqref{nemp}.
\end{proof}

Finally, we use Lemma \ref{L1} to establish the opposite inequality to the one
in Lemma \ref{VL}.

\begin{lem}\label{lem:ineq1}
	For any $\nu\in\cP(\T^n)$ and $t>0$, we have
	\[
	\int_{\T^n}u(z,t)\,d\nu(z)
	\ge 
	\inf_{\gam\in\cH(\nu;0,t)}
	\left\{
	\int_{\T^n\times\R^n\times[0,t]}L(z,q)\,d\gam(z,q,s)+\int_{\T^n}u_0(z)\,d\nu^\gam(z)
	\right\}. 
	\] 
\end{lem}

\begin{proof}
	For $s\in[0,t]$, let $\gam^\ep_s$ be the measure constructed in the proof of Lemma \ref{L1}
	for $t_0=0$ and $t_1=t$.
	By the definition of Legendre transform in \eqref{legendre}, we have 
	\[
	L(z,q)=D_pH(z, D_qL(z,q))D_qL(z,q)-H(z, D_qL(z,q)).
	\]
	Therefore,  
	\begin{align*}
	&\int_{\T^n \times \R^n} L(z,q)\,d\gam^{\ep}_s(z,q)\\
	=&\,
	\int_{\T^n \times \R^n} 
	(D_pH(z,D_qL(z,q))\cdot D_qL(z,q)-H(z,D_qL(z,q)))\,d\gam^{\ep}_s(z,q)\\ 
	=&\,
	\int_{\T^n \times \R^n} 
	(D_pH(z,p)\cdot p -H(z,p))\,d\beta^{\ep}_s(z,p)\\
	=&\,
	\int_{\T^n} 
	(D_pH(x,Du^{\ep})\cdot Du^\ep -H(x,Du^{\ep}))\sig^{\ep}(x,s)\,dx 
	\end{align*}
	for all $s\in[0,t]$. 
	Moreover, integrating by parts and using the adjoint equation, \eqref{eq:adjoint},  and \eqref{eq:ap}, we obtain
	\begin{align*}
	&
	\int_{\T^n} 
	(D_pH(x,Du^{\ep})\cdot Du^{\ep} -H(x,Du^{\ep}))\sig^{\ep}\,dx\\
	=&\,
	\int_{\T^n} 
	-\div(D_pH(x,Du^{\ep})\sig^{\ep})u^{\ep}-H(x,Du^{\ep})\sig^{\ep}\,dx\\
	=&\,
	\int_{\T^n} 
	(\sig^{\ep}_t+\ep\Del\sig^\ep+\Del(a\sig^\ep))u^{\ep}
	+(u^\ep_t-\ep\Del u^\ep-a \Del u^\ep)\sig^{\ep}\,dx\\
	=&\,
	\int_{\T^n} 
	(u^\ep\sig^{\ep})_t\,dx. 
	\end{align*}
	
Integrating on $[0,t]$ yields 
\begin{align*}
\int_{\T^n \times \R^n\times[0,t]} L(x,q)\,d\gam^{\ep}(x,q,s)
&=\,
\int_0^t\int_{\T^n \times \R^n} L(x,q)\,d\gam^{\ep}_s(x,q)ds\\
&=\, 
\int_{\T^n}u^\ep(x,t)\,d\nu
-\int_{\T^n} u_0(x)\,d\nu^\ep_0.  
\end{align*}	
Taking subsequences $\{\gam^{\ep_j}\}$ and $\{\nu_0^{\ep_j}\}$ as in \eqref{subseq1} yields  
\[
\int_{\T^n}u(z,t)\,d\nu(z)
=\int_{\T^n \times \R^n \times [0,t]} L(z,q)\,d\gam(z,q,s)+
\int_{\T^n} u_0(z)\,d\nu_0(z). 
\]
Because 
$\gam\in\cH(\nu_0,\nu;0,t)$,
we obtain the inequality claimed in the statement. 
\end{proof}

\begin{proof}[Proof of Theorem \ref{thm:rep}]
The statement follows directly by combining Lemma \ref{L1} with Lemma
	\ref{lem:ineq1}.
\end{proof}

\section{Properties of generalized holonomic measures}
\label{S3}

Here, we discuss
two properties of generalized holonomic measures in Lemmas \ref{lem:Mather-holonomic} and \ref{lem:connection}. Before proceeding, we recall that the results in \cite[Proposition 2.3]{MR3581314}, \cite[Lemma 2.1]{MT2} (see also the proof of Proposition \ref{prop:erg} in Appendix \ref{Appen:B}) imply that 
 the Mather set and the projected Mather set
are not empty; that is, 
\[
\widetilde\cM\not=\emptyset, \quad\text{ and }\quad \cM\not=\emptyset.
\]

\begin{lem}\label{lem:Mather-holonomic}
	Let $t>0$, $\mu\in\widetilde\cM$, and set $\nu:=\proj_{\T^n} \, \mu\in\cM$. 
	Define $d\gam(x,q,s):=d\mu(x,q)\,ds$ for all $s\in[0,t]$; that is, 
\begin{equation}\label{eq:M-hol}
	\int_{\T^n\times\R^n\times[0,t]}f(x,q,s)\,d\gam(x,q,s)
	=
	\int_0^{t}\int_{\T^n\times\R^n}f (x,q,s)\,d\mu(x,q)\,ds
%	=t\int_{\T^n\times\R^n}\,d\mu(x,q).
\end{equation}
for all $f\in C_c(\T^n\times\R^n\times[0,t])$.
Then, 
	\[
	\gam\in\cH(\nu, \nu; 0,t).
	\]  
\end{lem}
\begin{proof}
	For all $\varphi\in C^2(\T^n\times[0,t])$, we have 
	\begin{align*}
	&\int_{\T^n\times\R^n\times[0,t]}
	\left(\varphi_t+q\cdot D\varphi-a(x)\Delta \varphi \right)\, d\gam\\
	=&\, 
	\int_0^t\int_{\T^n\times\R^n}\varphi_t\,d\mu\, ds
	+\int_0^t\int_{\T^n\times\R^n}\left( q\cdot D\varphi-a(x)\Delta \varphi \right)\, d\mu\,ds\\
	=&\, 
	\int_{\T^n\times\R^n}\left(\int_0^t\varphi_t\,ds\right)\,d\mu
	=
	\int_{\T^n}\varphi(x,t)\,d\nu
	-
	\int_{\T^n}\varphi(x,0)\,d\nu,  
	\end{align*}
	which implies the conclusion. 
\end{proof}

Next, we show how to concatenate holonomic measures.

\begin{lem}\label{lem:connection}
	Let $\nu_1,\nu_2, \nu_3\in\cP(\T^n)$, and $a, b>0$ 
	be such that 
	$\cH(\nu_1,\nu_2;0,a)\not=\emptyset$, 
	$\cH(\nu_2,\nu_3;0,b)\not=\emptyset$. 
	For any $\gam_1\in\cH(\nu_1,\nu_2;0,a)$ and $\gam_2\in\cH(\nu_2,\nu_3;0,b)$, 
	set 
	\[
	\gam(x,q,s)
	:=
	\left\{
	\begin{array}{ll}
	\gam_1(x,q,s) & \quad s\in[0,a], \\
	\gam_2(x,q,s-a) & \quad s\in[a,a+b].  
	\end{array}
	\right. 
	\]
	Then, 
	$\gam\in \cH(\nu_1,\nu_3;0,a+b)$. 
\end{lem}
\begin{proof}
	By the definition of $\cH(\nu_1,\nu_2;0,a)$ and $\cH(\nu_2,\nu_3;0,b)$, 
	for any test function $\varphi\in C^2(\T^n\times [0, +\infty))$, we have 
	\begin{align*}
	&
	\int_{\T^n}\varphi(\cdot, a)\, d\nu_2- \int_{\T^n}\varphi(\cdot, 0)\, d\nu_1
	=\int_{\T^n\times\R^n\times[0,a]} \left(\varphi_t-a(x)\Delta \varphi+q\cdot D\varphi \right)\,d\gam_1, \\
	&
	\int_{\T^n}\varphi(\cdot, b)\, d\nu_3- \int_{\T^n}\varphi(\cdot, 0)\, d\nu_2
	=\int_{\T^n\times\R^n\times[0,b]} \left(\varphi_t-a(x)\Delta \varphi+q\cdot D\varphi \right)\,d\gam_2. 
	\end{align*}
	
	Thus, 
	\begin{align*}
	& 
	\int_{\T^n\times\R^n\times[0,a+b]} \left(\varphi_t-a(x)\Delta \varphi+q\cdot D\varphi \right)\,d\gam\\
	=&\, 
	\int_{\T^n\times\R^n\times[0,a]} \left(\varphi_t-a(x)\Delta \varphi+q\cdot D\varphi \right)\,d\gam 
	+\int_{\T^n\times\R^n\times[a,a+b]}\left(\varphi_t-a(x)\Delta \varphi+q\cdot D\varphi \right)\,d\gam\\
	=&\, 
	\int_{\T^n\times\R^n\times[0,a]}\left(\varphi_t-a(x)\Delta \varphi+q\cdot D\varphi \right)\,d\gam_1(x,q,s)\\
	&+\int_{\T^n\times\R^n\times[a,a+b]}\left(\varphi_t-a(x)\Delta \varphi+q\cdot D\varphi \right)\,d\gam_2(x,q,s-a)\\
	=&\, 
	\int_{\T^n}\varphi(\cdot, a)\, d\nu_2- \int_{\T^n}\varphi(\cdot, 0)\, d\nu_1\\
	&+\int_{\T^n\times\R^n\times[0,b]}\left( \varphi_t(x,s+a)-a(x) \Del \varphi(x,s+a)+q\cdot D\varphi(x,s+a) \right)\,d\gam_2(x,q,s)\\
	=&\, 
	\int_{\T^n}\varphi(\cdot, a)\, d\nu_2- \int_{\T^n}\varphi(\cdot, 0)\, d\nu_1
	+
	\int_{\T^n}\varphi(\cdot, a+b)\, d\nu_3- \int_{\T^n}\varphi(\cdot, a)\, d\nu_2\\
	=&\, 
	\int_{\T^n}\varphi(\cdot, a+b)\, d\nu_3- \int_{\T^n}\varphi(\cdot, 0)\, d\nu_1,  
	\end{align*}
	which implies 
	\[
	\gam\in \cH(\nu_1,\nu_3;0,a+b). 
	\qedhere
	\]
\end{proof}

\section{Peierls barrier and the generalized Aubry set}
\label{PBAS}

Suppose that Assumptions  \ref{A1}-\ref{A3} hold.
Let $u$ solve Problem \ref{P1}. As shown in \cite{CGMT}, there exists
a large-time limit, $u_\infty$, given by \eqref{uinfdef}. This limit function solves Problem \ref{E}.
As explained in the Introduction, by the results in \cite{MT2}, to characterize a solution of Problem \ref{E}, it is enough to determine the value 
\[
\int_{\T^n}u_{\infty}(z)\,d\nu(z) \quad\text{ for all} \ \nu\in\cM. 
\]
In this section, we provide this characterization and prove Theorem \ref{thm:profile}. 
Moreover, we examine two constructions from Aubry-Mather theory, the Peirls barrier and the projected Aubry set, and extend them in a way suitable for the study of degenerate diffusions.

Recall that, for $t>0$, we set
\begin{align*}
h_t(\nu_0,\nu_1)=
\left\{
\begin{array}{ll}
\displaystyle
\inf_{\gam\in\cH(\nu_0,\nu_1;0,t)}
\int_{\T^n\times\R^n\times[0,t]}L(x,q)\,d\gam 
& 
\quad \text{ if} \ \cH(\nu_0,\nu_1;0,t)\not=\emptyset, \\ 
+\infty & 
\quad \text{ if} \ \cH(\nu_0,\nu_1;0,t)=\emptyset,  
\end{array}
\right.  
\end{align*}
for $\nu_0, \nu_1\in \cP(\Tt^n)$.
Also,
\begin{equation}
\label{manep}
d(\nu_0,\nu_1) = \inf_{t>0} h_t(\nu_0,\nu_1).
\end{equation}

\begin{lem}
	Consider the setting of Problem \ref{P1} and
	suppose that Assumptions  \ref{A1}-\ref{A3} hold.
	Fix $t>0$ and $\nu_1\in\cP(\T^n)$. 
	Then, the map
	\[
	\nu\mapsto h_t(\nu,\nu_1)
	\] 
	is convex. 
\end{lem}
\begin{proof}
	Fix $t>0$ and $\nu_1\in\cP(\T^n)$. 
	Take $\nu_0, \bar \nu_0 \in \cP(\T^n)$.
	We must show that for $0\leq \lambda\leq 1$, 
	\[h_t(\lambda \nu_0+(1-\lambda)\bar \nu_0, \nu_1)\leq 
	\lambda h_t( \nu_0, \nu_1)+(1-\lambda) h_t( \bar \nu_0, \nu_1).
	\]
	If any of the terms in the right-hand side is $+\infty$, the result is trivial. Thus, 
	we may assume that $h_t( \nu_0, \nu_1)<+\infty$ and
	$h_t( \bar \nu_0, \nu_1)<+\infty$. Accordingly, 
	 $\cH(\nu_0, \nu_1; 0,t)\not=\emptyset$ and $\cH(\bar \nu_0, \nu_1; 0,t)\not=\emptyset$. 
	For $\gam_1\in \cH(\nu_0, \nu_1; 0,t)$ and $\gam_2\in \cH(\bar \nu_0, \nu_1; 0,t)$, let
	$\gam:=\lam\gam_1+(1-\lam)\gam_2$ for $\lam\in[0,1]$. 
	We claim that
	\begin{equation}
	\label{gin}
	\gam\in\cH(\lam\nu_0+(1-\lam)\bar \nu_0,\nu_1; 0,t). 
	\end{equation}
	To establish the claim, we fix 
$\varphi\in C^2(\T^n\times[0,t])$. Then
	\begin{align*}
	&  
	\int_{\T^n\times\R^n\times[0,t]}
	\left(\varphi_t-a\Del \varphi+q\cdot D\varphi\right)\,d(\lam\gam_1+(1-\lam)\gam_2) \\
	=\, &  
	\lam\left(
	\int_{\T^n\times\R^n\times[0,t]}
	\left(\varphi_t-a\Del \varphi+q\cdot D\varphi \right)\,d\gam_1
	\right)\\
	&
	+
	(1-\lam)
	\left(
	\int_{\T^n\times\R^n\times[0,t]}
	\left( \varphi_t-a\Del \varphi+q\cdot D\varphi \right)\,d\gam_2
	\right)\\
	=\, &  
	\lam\left(
	\int_{\T^n}\varphi(z,t)\,d\nu_1(z)
	-\int_{\T^n}\varphi(z,0)\,d\nu_0(z)  
	\right)\\
	&
	+(1-\lam)\left(
	\int_{\T^n}\varphi(z,t)\,d\nu_1(z)
	-\int_{\T^n}\varphi(z,0)\,d\bar \nu_0(z)  
	\right)\\
	=\,&  
	\int_{\T^n}\varphi(z,t)\,d\nu_1(z)
	-\int_{\T^n}\varphi(z,0)\,d(\lam\nu_0+(1-\lam)\bar\nu_0)(z),
	\end{align*}
	thus \eqref{gin} holds. 
	Accordingly, we have 
	\begin{align*}
	h_t(\lam\nu_0+(1-\lam)\bar \nu_0,\nu_1)
	\le 
	\lam \int_{\T^n\times\R^n\times[0,t]}L\,d\gam_1
	+(1-\lam)\int_{\T^n\times\R^n\times[0,t]}L\,d\gam_2.   
	\end{align*}
	Taking the infimum on $\gam_1\in \cH(\nu_0,\nu_1;0,t)$ and $\gam_2\in \cH(\bar \nu_0,\nu_1;0,t)$,
	we obtain 
	\[
	h_t(\lam\nu_0+(1-\lam)\bar\nu_0,\nu_1)\le 
	\lam h_t(\nu_0,\nu_1)
	+(1-\lam)h_t(\bar \nu_0,\nu_1). 
	\qedhere
	\]
\end{proof}

\medskip

Next, we define the {\em generalized Peierls barrier}, $h:\cP(\T^n)\times\cP(\T^n)\to\R\cup\{+\infty\}$, by
\begin{equation*}%\label{func:h}
h(\nu_0,\nu_1)=\liminf_{t\to\infty}h_t(\nu_0,\nu_1) \quad \text{ for } \nu_0, \nu_1 \in \cP(\T^n).  
\end{equation*}
In weak KAM theory, the Peierls barrier is a function
$h: \Tt^n\times \Tt^n\to\R$. However, our context, points in $\Tt^n$
are replaced by probability measures in $\cP(\Tt^n)$; thus, 
$h$ becomes a function on $\cP(\Tt^n)\times \cP(\Tt^n)$.

Now, we introduce an auxiliary function, $m:\cP(\T^n)\times\cP(\T^n)\to\R$, given 
by
\begin{equation*}%\label{func:m}
m(\nu_0,\nu_1):=
\sup_{w\in \Ss}\left[\int_{\T^n}w\,d\nu_1(x)- \int_{\T^n}w\,d\nu_0(x)\right],
\end{equation*}
where $\Ss$ is the set of all  Lipschitz viscosity solutions to \eqref{ee} given by \eqref{def:Ss}. 
The classical result for the existence of a solution of the ergodic problem implies that 
$\Ss\not=\emptyset$ (see \cite{LMT}, for instance). Thus, 
$m(\nu_0,\nu_1)$ is finite for all $\nu_0,\nu_1\in\cP(\T^n)$.

\begin{lem}\label{lem:m}
	Consider the setting of Problem \ref{E}  and 
	suppose that Assumptions  \ref{A1}-\ref{A3} hold.
	Then, 
	\begin{equation}\label{ineq:md}
	m(\nu_0,\nu_1)\le d(\nu_0,\nu_1)
	\quad 
	\text{ for all} \ \nu_0,\nu_1\in\cP(\T^n). 
	\end{equation}
\end{lem}

\begin{proof}
	If $d(\nu_0, \nu_1)=\infty$, then \eqref{ineq:md} is trivial. Therefore, 
	we only need to address the case where $d(\nu_0, \nu_1)<\infty$. 
	In this case, there exists $t>0$ such that 
	$\cH(\nu_0,\nu_1; 0, t)\not=\emptyset$. 
	Let $w$ be any Lipschitz viscosity solution to \eqref{ee}, and 
	let $w^{\al}\in C^\infty(\T^n)$ be the approximation of $w$
	given by \eqref{func:w-al} for $\alpha>0$; that is,
	\[
	w^\al(x)= \int_{\T^n} \theta^\al(y)w(x-y)\,dy \quad \text{ for $x \in\T^n$,}
	\]
	since $w$ is time-independent here.
	Then, $w^\al \in C^\infty(\T^n)$ satisfies
	\[
	H(x,Dw^\al) - a(x)\Del w^\al \leq C \al^{1/2} \quad \text{ in } \T^n.
	\]

For any $\gam\in\cH(\nu_0, \nu_1;0,t)$, we have 
	\begin{align*}
	&\int_{\T^n}w^\al(x)\, d\nu_1(x)-\int_{\T^n}w^\al(x)\, d\nu_0(x)\\
	= &\, 
	\int_{\T^n\times\R^n\times[0,t]}
	\left(q\cdot Dw^\al(x)-a(x)\Del w^\al(x) \right)\, d\gam(x,q,s)\\
	\le &\,
	\int_{\T^n\times\R^n\times[0,t]}\left(L(x,q)+H(x,Dw^\al(x))-a(x)\Del w^\al(x)\right)\, d\gam(x,q,s)\\ 
	\le &\, 
	\int_{\T^n\times\R^n\times[0,t]}L(x,q)\, d\gam(x,q,s)
	+ C \al^{1/2} t.
	\end{align*}
	Send $\al\to0$   to yield
	\[
	\int_{\T^n}w(x)\, d\nu_1(x)-\int_{\T^n}w(x)\, d\nu_0(x)
	\le 
	\int_{\T^n\times\R^n\times[0,t]}L(x,q)\, d\gam(x,q,s). 
	\]
	Thus, taking the supremum with respect to $w \in \Ss$ gives the 
claim in the statement. 
\end{proof}

\begin{cor}
\label{CC1}	
Consider the setting of Problem \ref{E}  and 
suppose that Assumptions  \ref{A1}-\ref{A3} hold. Then, for all $\nu_0, \nu_1\in \cP(\T^n)$, 
\begin{equation}
\label{mltd}
m(\nu_0,\nu_1)\le
d(\nu_0,\nu_1)\le h(\nu_0,\nu_1).
\end{equation}
Moreover,	
$h(\nu, \nu)\geq 0$ for all $\nu \in \cP(\T^n)$.
\end{cor}
\begin{proof}
Because of the definition of $d(\nu_0,\nu_1)$ in \eqref{def:d} and \eqref{manep}, it is clear that 
$d(\nu_0,\nu_1)\le h(\nu_0,\nu_1)$.
Lemma \ref{lem:m}  gives $m(\nu_0,\nu_1)\le
d(\nu_0,\nu_1)$.
Hence, we have \eqref{mltd}. 
Finally, because $m(\nu,\nu)=0$,
we conclude that $h(\nu, \nu)\geq 0$.  
\end{proof}

The {\em generalized projected  Aubry set}
is the set
	\begin{equation*}%\label{def:Aubry}
	\cA:=\{\nu\in\cP(\T^n)\,:\, h(\nu,\nu)=0\}. 
	\end{equation*}
As in standard Aubry-Mather theory, the generalized projected Aubry set contains the generalized projected Mather set, as we establish next.

\begin{pro}\label{prop:A-M}
	Consider the setting of Problem \ref{E}  and 
     suppose that Assumptions \ref{A1}-\ref{A3} hold.
	Then, 
	\[
	\cM\subset\cA. 
	\] 
	In particular, 
	\[
	\cA\not=\emptyset. 
	\]
\end{pro}
\begin{proof}
By Corollary \ref{CC1}, 
	\[
	0\le h(\nu,\nu)\quad\text{ for all} \ \nu\in\cP(\T^n). 
	\]
	
	Let $\mu$ be any generalized Mather measure, and set 
	$\nu:=\proj_{\T^n}\,\mu$. 
	Let
 $d\tilde{\gam}(x,q,s):=d\mu(x,q)\,ds$ for all $s\in[0,t]$. 
 	By Lemma \ref{lem:Mather-holonomic}, 
  $\tilde{\gam}\in\cH(\nu,\nu;0,t)$ for all $t>0$. 
	Thus, 
	\begin{align*}
	h(\nu,\nu)
	=&\, 
	\liminf_{t\to\infty}h_t(\nu,\nu)
	=
	\liminf_{t\to\infty} 
	\inf_{\gam\in\cH(\nu,\nu;0,t)}
	\int_{\T^n\times\R^n\times[0,t]}L(x,q)\,d\gam(x,q,s)\\
	\le&\, 
	\liminf_{t\to\infty} 
	\int_{\T^n\times\R^n\times[0,t]}L(x,q)\,d\tilde{\gam}(x,q,s)
	=
	\liminf_{t\to\infty} 
	\int_0^t\int_{\T^n\times\R^n}L(x,q)\,d\mu\, ds\\
	=&\, 
	0, 
	\end{align*}
in light of \eqref{eq:erg-const}, 	which finishes the proof. 
	Because of the results in \cite{MR3581314, MT2} (see also the proof of Proposition \ref{prop:erg}), we have $\cM\not=\emptyset$. 
	Thus, $\cA\not=\emptyset$. 
\end{proof}

Now, we prove a stronger version of Theorem \ref{thm:profile}; that is, \eqref{maver} holds
for all measures in $\Aa$.

\begin{teo}\label{thm:profile-A}
Suppose that Assumptions  \ref{A1}-\ref{A3} hold. 
Let $u$ solve Problem \ref{P1} and $u_\infty$ be the large time limit of $u$ in \eqref{uinfdef}.  
Then,
	\[
	\int_{\T^n} u_\infty(z) \,d\nu(z) = \inf_{\nu_0\in\cP(\T^n)} 
	\left[d(\nu_0,\nu) + \int_{\T^n} u_0(z)\,d\nu_0(z)  \right]
	\]
	for all $\nu\in\cA$.
\end{teo}
\begin{proof} 
We rewrite the representation formula in Theorem \ref{thm:rep} as  
\[
\int_{\T^n} u(z,t)\, d\nu(z) = \inf_{\substack{\gam \in \cH(\nu_0,\nu;0,t)  \\ \nu_0\in \cP(\Tt^n)}} 
	\left[\int_{\T^n \times \R^n \times [0,t]} L(x,q) \, d\gam(x,q,s) + \int_{\T^n} u_0(z) \,d\nu_0(z)\right]. 
\]
Accordingly, we have
\begin{equation*}%\label{ineq:profile}
\int_{\T^n} u(z,t)\, d\nu(z)\ge
\inf_{\substack{\gam \in \cH(\nu_0,\nu;0,t)  \\ \nu_0\in \cP(\Tt^n)}} \left[d(\nu_0,\nu) + \int_{\T^n} u_0 \,d\nu_0\right] 
= 
\inf_{\nu_0\in\cP(\T^n)} \left[d(\nu_0,\nu) + \int_{\T^n} u_0\,d\nu_0 \right]
\end{equation*}
for all $\nu\in\cP(\T^n)$. Thus, by letting $t\to \infty$, we get 
\[
\int u_\infty(z) \,d\nu(z) \geq \inf_{\nu_0\in\cP(\T^n)} \left[d(\nu_0,\nu) + \int_{\T^n} u_0(z)\,d\nu_0(z)  \right].
\]

Let $\nu\in\cA$. Fix $\ep>0$, and 
	take $\nu_0\in\cP(\T^n)$, $T>0$, and $\gam_1\in \cH(\nu_0,\nu;0,T)$ such that
\begin{align*}
&\inf_{\nu_0\in\cP(\T^n)} \left[d(\nu_0,\nu) + \int_{\T^n} u_0(z)\,d\nu_0(z)  \right] + \ep \\
\geq& \int_{\T^n \times \R^n \times [0,T]} L(z,q) \, d\gam_1(z,q,s) + \int_{\T^n} u_0(z) \,d\nu_0(z) \geq \int_{\T^n} u(z,T) \, d\nu(z) . 
\end{align*}
Because $\nu\in\cA$, there exists a sequence $\{t_k\}_{k\in\N}$ such that $t_k\to\infty$ and 
	\[
	h_{t_k}(\nu,\nu)\to \liminf_{t\to \infty } 	h_{t}(\nu,\nu)=h(\nu,\nu)=0 \quad\text{as} \ k\to\infty. 
	\]
	Accordingly, there exists $k_0\in\N$ such that for all $k\ge k_0$, 
	\[
	\ep\ge h_{t_k}(\nu, \nu). 
	\]
	Moreover, there exists a corresponding measure, $\gam_k\in\cH(\nu,\nu;0,t_k)$,
	such that 
	\[
	2\ep\ge h_{t_k}(\nu, \nu)+\ep
	\ge 
	\int_{\T^n \times \R^n \times [0,t_k]} L(x,q) \, d\gam_k(x,q,s). 
	\]
We note that $u(\cdot,\cdot+T)$ solves \eqref{eq:1} with initial data $u(\cdot,T)$.
By the representation formula in Theorem \ref{thm:rep}, for $k \geq k_0$,
	\begin{align*}
	&\inf_{\nu_0\in\cP(\T^n)} \left[d(\nu_0,\nu) + \int_{\T^n} u_0\,d\nu_0  \right]+\ep
	\ge 
	\int_{\T^n \times \R^n \times [0,T]} L(x,q) \, d\gam_1(x,q,s) + \int_{\T^n} u_0 \,d\nu_0\\
	\ge &\,
	\int_{\T^n} u(z,T) \, d\nu(z) 
	+ 
	\int_{\T^n \times \R^n \times [0,t_k]} L(x,q) \, d\gam_k(x,q,s)-2\ep\\
	 \ge &\,
	\int_{\T^n} u(x,T+t_k)\, d\nu-2\ep. 
	\end{align*}
	Letting $k\to \infty$ and then $\ep \to 0$ gives
\[
\int_{\T^n}  u_\infty \, d\nu \leq \inf_{\nu_0\in\cP(\T^n)} \left[d(\nu_0,\nu) + \int_{\T^n} u_0\,d\nu_0  \right],
\]
which finishes the proof. 
\end{proof}

Theorem  \ref{thm:profile} follows from the preceding result, as we point out next. 

\begin{proof}[Proof of Theorem  \ref{thm:profile}]
	According to Proposition \ref{prop:A-M}, $\Mm\subset \Aa$. 
	Thus, the statement follows directly from Theorem
	\ref{thm:profile-A}. 
\end{proof}

Next, we record a convexity property for the generalized Ma\~n\'e
critical potential. 
We notice that we do not know if the maps $\nu\mapsto d(\nu,\bar \nu)$ and
$\nu\mapsto d(\bar \nu,\nu)$ are convex for all $\bar \nu\in \cP(\Tt^n)$.
We can only prove that the maps $\nu\mapsto d(\nu,\bar \nu)$ and
$\nu\mapsto d(\bar \nu,\nu)$ are convex for all $\bar \nu \in \Mm$.

\begin{lem}
	Consider the setting of Problem \ref{P1} and suppose that Assumptions \ref{A1}-\ref{A3} hold. Then, 
    for any projected Mather measure $\bar \nu\in \Mm$, the maps
	\[
	\nu\mapsto d(\nu,\bar \nu), 
	\quad 
	\nu\mapsto d(\bar \nu,\nu) 
	\] 
	are convex for $\nu\in\cP(\T^n)$. 
\end{lem}
\begin{proof}
	We only prove that the map $\nu\mapsto d(\bar \nu,\nu) $ is convex. Since the proof is similar for the other map, we omit it here. 
	
	Take $\nu_1, \nu_2\in\cP(\T^n)$. We claim that for $0\leq \lambda\leq 1$,
	\[
	d(\bar \nu,\lambda \nu_1+(1-\lambda )\nu_2) \leq
	\lambda d(\bar \nu,\nu_1)
	+
	(1-\lambda ) d(\bar \nu,\nu_2).
	\]
	It is enough to consider the case where both terms in the right-hand side of
	the preceding expression are finite. 
	Hence, let $t, s>0$ with $t<s$ so that  
	$\cH(\bar \nu,\nu_1; 0,t)\not=\emptyset$ and $\cH( \bar \nu,\nu_2; 0,s)\not=\emptyset$. 
	Let 
	$\gam_1\in \cH( \bar \nu, \nu_1; 0,t)$ and $\gam_2\in \cH( \bar \nu, \nu_2; 0,s)$.

	Since $ \bar \nu$ is a projected Mather measure, we can extend $\gam_1$ to $\bar \gam_1\in\cH( \bar \nu, \nu_1; 0,s)$ 
by a similar argument to the proof of Lemma \ref{lem:Mather-holonomic}. 
	Indeed, let $\mu$ be a Mather measure such that $ \bar \nu=\proj_{\T^n} \mu$.
Let
	\[
	d\bar \gam_1(x,q,r):=
	\begin{cases}
	d \mu(x,q,r) \qquad &\text{ for } r \in [0,s-t],\\
	d \gam_1(x,q,r-(s-t)) \qquad &\text{ for } r \in [s-t,s].
	\end{cases}
	\]
	By abuse of notation, we identify $\bar \gam_1$ as $\gam_1$.
	Let
	$\gam:=\lam\gam_1+(1-\lam)\gam_2$ for $\lam\in[0,1]$.
	As we show next, 
	\[
	\gam\in\cH(\bar \nu, \lam\nu_1+(1-\lam)\nu_2; 0,s). 
	\]
To verify the preceding claim, fix $\varphi\in C^2(\T^n\times[0,s])$, 
	\begin{align*}
	&  
	\int_{\T^n\times\R^n\times[0,s]}
	\left(\varphi_t-a(x)\Del \varphi+q\cdot D\varphi\right)\,d(\lam\gam_1+(1-\lam)\gam_2) \\
	=\,&  
	\lam\left(
	\int_{\T^n\times\R^n\times[0,s]}
	\left(\varphi_t-a(x)\Del \varphi+q\cdot D\varphi \right)\,d\gam_1
	\right)\\
	&
	+
	(1-\lam)
	\left(
	\int_{\T^n\times\R^n\times[0,s]}
	\left(\varphi_t-a(x)\Del \varphi+q\cdot D\varphi \right)\,d\gam_2
	\right)\\
	=\,&  
	\lam\left(
	\int_{\T^n}\varphi(z,s)\,d\nu_1(z)
	-\int_{\T^n}\varphi(z,0)\,d \bar \nu(z)  
	\right)\\
	&
	+(1-\lam)\left(
	\int_{\T^n}\varphi(z,s)\,d\nu_2(z)
	-\int_{\T^n}\varphi(z,0)\,d \bar \nu(z)  
	\right)\\
	=\,&  
	\int_{\T^n}\varphi(z,s)\,d(\lam\nu_1+(1-\lam)\nu_2)(z)
	-\int_{\T^n}\varphi(z,0)\,d \bar \nu(z). 
	\end{align*}
	
	Thus, we have 
	\begin{align*}
	&d( \bar \nu,\lam\nu_1+(1-\lam)\nu_2)
	\le\,  
	\lam \int_{\T^n\times\R^n\times[0,s]}L\,d\gam_1
	+(1-\lam)\int_{\T^n\times\R^n\times[0,s]}L\,d\gam_2\\
	&=\,  
	\lam \int_{\T^n\times\R^n\times[0,s-t]}L\,d\mu
	+\lam \int_{\T^n\times\R^n\times[s-t,s]}L\,d\gam_1
	+(1-\lam)\int_{\T^n\times\R^n\times[0,s]}L\,d\gam_2\\
	&=\,  
	\lam \int_{\T^n\times\R^n\times[0,t]}L\,d\gam_1
	+(1-\lam)\int_{\T^n\times\R^n\times[0,s]}L\,d\gam_2	
	\end{align*}
	because of \eqref{eq:erg-const}. 
	Taking the infima on 
	$\gam_1\in\cH(\bar \nu,\nu_1;0,t)$ for $t>0$ 
	and $\gam_2\in\cH( \bar \nu,\nu_2;0,s)$ for  $s>0$, 
	respectively, 
	we obtain 
	\[
	d( \bar \nu,\lam\nu_1+(1-\lam)\nu_2)\le 
	\lam 
	d( \bar \nu,\nu_1)+(1-\lam)d( \bar \nu,\nu_2). 
\qedhere	\]
\end{proof}

We end this section by establishing the improved convergence result in Proposition \ref{conv}. 
This result, which is of independent interest, is crucial for the proof of the duality formula for the Ma\~n\'e critical potential. 
\begin{proof}[Proof of Proposition \ref{conv}]
	We establish the claim by contradiction. Accordingly, we assume that 
	there exists  a sequence $\{u_0^k\} \subset C(\T^n)$ with $\min_{\T^n}u_0^k=0$ and a positive number, $\ep>0$, 
	such that for all $k\in \N$ there exists $t_k>k$ with
	\[
	\|u^k(\cdot,t_k)-u_\infty^k\|_{L^\infty(\T^n)}>\ep.
	\]
	Because of Assumption \ref{A4}, 
	extracting a subsequence if
	necessary,  we may assume that  $u^k(\cdot,1)\to u(\cdot,1)$ uniformly on $\T^n$ as $k \to \infty$, for some function $u(\cdot,1) \in C(\T^n)$. Let $u_\infty$ be the large time limit of the solution  to \eqref{eq:1} with initial data $u(\cdot,1)$.

	Because $u^k(\cdot,1)\to u(\cdot,1)$ uniformly on $\T^n$, we have $u_\infty^k\to u_\infty$ uniformly on $\T^n$.
	By choosing $k$ large enough, we have
	$\|u^k(\cdot,t)- u(\cdot,t)\|_{L^\infty(\T^n)}\leq \frac \ep 3$ for all $t \geq 1$, and $\|u_\infty^k- u_\infty\|_{L^\infty(\T^n)}\leq \frac \ep 3$.
	Then,
	\begin{align*}
	&\|u(\cdot,t_k)-u_\infty\|_{L^\infty(\T^n)}\\
	\geq \ &
	- \|u(\cdot,t_k)-u^k(\cdot,t_k)\|_{L^\infty(\T^n)}
	+
	\|u^k(\cdot,t_k)-u^k_\infty\|_{L^\infty(\T^n)}
	-
	\|u_\infty^k- u_\infty\|_{L^\infty(\T^n)} \\
	\geq \ &\frac \ep 3, 
	\end{align*}
	which contradicts the uniform convergence of $u(\cdot,t)$ to $u_\infty$ in $\T^n$ as $ t\to \infty$.
\end{proof}

%%%%%%%%%%%%%%

\section{Duality formulas}
\label{rdualty}

In this section, 
we use the Legendre-Fenchel-Rockafellar theorem to give dual representations
for $h_t$  and $d$.

\subsection{Duality}\label{subsec:dual}

To use the Legendre-Fenchel-Rockafellar theorem, we need to recall the definition of Legendre transform
on a locally convex topological vector space. 
Let $E$
be a locally convex topological vector space with dual
$E'$ and 
 duality pairing 
$(\cdot,\cdot)$.
Consider a convex function, 
$f:E\rightarrow (-\infty,+\infty]$.
The
{\em Legendre-Fenchel transform}, $f^*:E'\rightarrow [-\infty,+\infty]$,
of $f$ is
\[
f^*(y)=\sup_{x\in E} \bigl((x,y)-f(x)\bigr) \quad \text{ for $y\in E'$.}
\]
Let $\Omega= \Tt^n\times \Rr^n\times [0,t]$ for $t>0$ fixed.
Let $m$ be as in Assumption \ref{A3}
and choose $\zeta$ with $1<\zeta<m'$.
Let $\Uu$ be the set
%be the set of Radon measures in $\Omega$ with weight 
%$\gamma$, that is, 
\[
\Uu=\left\{\mu\in \cR(\Omega) \,:\,
\int_\Omega (1+|q|^\zeta) \, d|\mu|<\infty\right\}.
\]
The set $\Uu$ is the dual of %the set $C_{\zeta}(\Omega)$ defined by 
\begin{align*}
%\label{qjo} 
C_{\zeta}(\Omega)
&:= \Big\{\phi\in C(\Omega) \,:\, 
\|\phi\|_\zeta:=\sup_{\Omega} \left|\frac{\phi(x,q,s)}{1+|q|^\zeta}\right|<\infty,\\
&\qquad 
\lim_{|q|\rightarrow \infty}\frac{\phi(x,q,s)}{1+|q|^\zeta} = 0 \quad \text{ uniformly for}\ (x,s)\in \Tt^n\times [0,t]
\Big\}.
\end{align*}
Let
\[
\Uu_1=\left\{\mu\in\Uu \,:\, \mu\geq 0, \ \int_\Omega d\mu = t\right\}.
\]

Next, for $q\in \Rr^n$, define the following linear operator on $C^2( \Tt^n\times [0,t])$
\[
A^q\varphi=	\varphi_t -a(x)\Del \varphi +q\cdot D\varphi.
\]
Note that $A^q$ can be regarded as a bounded linear mapping from  $C^2( \Tt^n\times [0,t])$
to $C_\zeta(\Omega)$.
Let $B:C( \Tt^n\times [0,t]) \to \Rr$ be a bounded linear operator.
Define
\[
\Uu_2=\cl\left\{\mu\in \Uu:\int_\Omega A^q\varphi \, d\mu(x,q,s)=B\varphi,  \,
\forall \varphi\in C^2(\Tt^n\times [0,t])\right\},
\]
where the closure, $\cl$, is taken with respect to the weak
topology in $\Uu$. Moreover, 
we assume that there exists $\bar \mu\in \Uu_2$; that is, $ \Uu_2$ is non-empty.

For $\phi\in C_{\zeta}(\Omega)$, let
\begin{equation}
\label{hdef}
f(\phi)=t \sup_{(x,q,s)\in \Omega}
(\phi(x,q,s)-L(x,q)).
\end{equation}
Because $f$
is the supremum of convex functions, it is also a convex function.
Moreover,  because
	\[
\lim_{|q|\to \infty}\frac{L(x,q)}{1+|q|^\zeta}=+\infty
\]
for all $x\in \Tt^n$, $f$
is also continuous with respect to the uniform convergence
in $C_{\zeta}(\Omega)$.

Let
\[
\Cg=\cl\left\{\phi\in C_\zeta(\Omega)\,:\,\phi(x,q,s)=A^q\varphi(x,s) \ \text{for some} \ \varphi\in  C^2(\Tt^n\times[0,t])\right\},
\]
where  $\cl$ denotes the closure in $C_{\zeta}$.
Because
$A^q$ is a bounded linear mapping from  $C^2( \Tt^n\times [0,t])$
to $C_\zeta(\Omega)$, $\Cg$
is a closed convex set.

Fix $\bar \mu \in \Uu_2$ and let $g:C(\Omega)\to\R\cup\{+\infty\}$ be
\begin{equation}
\label{gdef}
g(\phi)=
\begin{cases}
-\int_{\Omega}  \phi \,d\bar \mu \quad&\text{ if } \phi\in\Cg,\\
+\infty&\text{ otherwise}.
\end{cases}
\end{equation}
Because $\Cg$ is a closed convex set, $g$ is convex and
lower semicontinuous. Furthermore, for $\phi=A^q \varphi$, we have 
$\int_{\Omega} \phi \, d\bar \mu=B\varphi $, according to the definition of $\Uu_2$.

Next, we address the Legendre transform of $f$. We recall that 
\[
f^*(\mu)=\sup_{\phi\in C_{\zeta}(\Omega)} \left(\int_{\Omega} \phi \,d\mu - f(\phi)\right).
\]
We first give two elementary results on $f^*$. 
\begin{lem}
	\label{ll1}
Suppose that Assumptions \ref{A1}-\ref{A3} hold. 	
Let $f$ be as in \eqref{hdef}. Let $\mu\in \Uu$. 
If $\mu\not \geq 0$ then $f^*(\mu)=+\infty$.
\end{lem}
\begin{proof}
	If $\mu\not \geq 0$, there exists a  sequence of functions
	$\phi_n\in C_\zeta(\Omega)$ with $\phi_n\leq 0$ such that
	$$
	\int_\Omega \phi_n\, d\mu\rightarrow +\infty.
	$$
	Therefore, because
	$$
	f(\phi_n)=
	t \sup_{(x,q,s)\in \Omega} \left[\phi_n(x,q,s)-L(x,q)\right]\leq -t L(0,0)\leq C,
	$$
	we have
	\[  
	f^*(\mu)\geq\lim_{n \to \infty} 	\int_\Omega \phi_n \, d\mu  -f(\phi_n)\ge 
	\lim_{n \to \infty} 	\int_\Omega \phi_n \, d\mu  -C =+\infty.
	\qedhere
	\]
\end{proof}
Next, we give a lower bound for $f^*(\mu)$ for non-negative $\mu$.
\begin{lem}
	\label{ll2}
	Suppose that Assumptions \ref{A1}-\ref{A3} hold. 
	Let $f$ be as in \eqref{hdef}. Let $\mu\in \Uu$.
	If $\mu\geq 0$ then
	$$
	f^{*}(\mu)\geq \int_\Omega L \, d\mu +\sup_{\psi\in C_{\zeta}(\Omega)}\left(
	\int_{\Omega}\psi \, d\mu-t \sup_{\Omega} \psi
	\right).
	$$
\end{lem}
\begin{proof} 
	Let $L_n$ be an increasing sequence of   functions in $C_\zeta(\Omega)$ such that $L_n \to L$ pointwise.
	Any
	$\phi$ in $C_\zeta(\Omega)$
	can be written as
	$\phi=\psi+L_n$, for some $\psi$ in $C_\zeta(\Omega)$.
	Therefore,
	\begin{multline*}
	\sup_{\phi\in C_\zeta(\Omega)}
	\left(\int_{\Omega}\phi \, d\mu-f(\phi)\right)=
	\sup_{\psi\in C_\zeta(\Omega)}
	\left(
	\int_{\Omega} L_n \, d\mu+\int_{\Omega} \psi \, d\mu-t \sup_{\Omega} (L_n+\psi-L)
	\right).
	\end{multline*}
	Because
	$L_n\leq L$,
	we have
	\begin{align*}
	\sup_{\phi\in C_\zeta(\Omega)}
	\left(\int_{\Omega} \phi \, d\mu-f(\phi)\right)& \geq 
	\sup_{\psi\in C_\zeta(\Omega)}
	\left(
	\int_{\Omega}  L_n \, d\mu+\int_{\Omega}  \psi \, d\mu-t \sup_{\Omega} \psi
	\right)\\
	&=\int_{\Omega}  L_n \, d\mu +
	\sup _{\psi\in C_\zeta(\Omega)} \left(
	\int_{\Omega}  \psi \, d\mu - t \sup_{\Omega} \psi
	\right)
	\end{align*}
	By the monotone convergence theorem, 
	\[
	\int_{\Omega} L_n \, d\mu\rightarrow \int_{\Omega} L \, d\mu
	\quad\text{as} \ n\to\infty. 
	\]
	Thus,
	\begin{align*}
f^\ast(\mu)=\sup_{\phi\in C_\zeta(\Omega)}
	\left(\int_{\Omega}\phi \, d\mu-f(\phi)\right)  \geq 
	\int_{\Omega} L\, d\mu+\sup_{\psi\in C_\zeta(\Omega)}
	\left(\int_{\Omega} \psi \, d\mu-t\sup \psi
	\right),
	\end{align*}
	as required.
\end{proof}

Using the two preceding lemmas, we now compute the Legendre transform of 
$f$.
\begin{pro}
\label{fstarp}
Suppose that Assumptions \ref{A1}-\ref{A3} hold. 
Let $f$ be as in \eqref{hdef}. Let $\mu\in \Uu$. Then, 
\begin{equation}
\label{hstar}
	f^*(\mu)=
	\begin{cases}
	\int_{\Omega} L\, d\mu \quad &\text{ if } \quad \mu\in\Uu_1,\\
	+\infty  &\text{ otherwise.}
	\end{cases}
\end{equation}
\end{pro}
\begin{proof}
By Lemma \ref{ll1}, 
 if 
	$\mu$ is  not non-negative
	then $f^*(\mu)=\infty$.
Moreover, by Lemma \ref{ll2}, if $\int_{\Omega} L\,d\mu=+\infty$ then $f^*(\mu)=+\infty$.
Furthermore, if $\int_{\Omega} \, d\mu\neq t$ then
	$$
	\sup_{\psi\in C_\zeta(\Omega)}\left(
	\int_{\Omega}\psi \,d\mu-t  \sup_{\Omega} \psi\right)\geq
	\sup_{\alpha\in \Rr} \alpha\left(\int_{\Omega} \,d\mu-t\right)=+\infty,
	$$
	by choosing $\psi \equiv \alpha$ for $\alpha\in \Rr$. 
	Therefore, in this case, using also Lemma \ref{ll2},  we have $f^*(\mu)=+\infty$.
	
	When $\mu\geq 0$ and $\int_{\Omega}\, d\mu=t$, 
	%and $\sup_{\psi\in C_\zeta(\Omega)}\left(
	%\int_{\Omega}\psi \,d\mu-t  \sup_{\Omega} \psi\right)\le0<+\infty$, 
	Lemma \ref{ll2} implies
	$$
	f^*(\mu)\geq \int_{\Omega} L\,d\mu,
	$$
	by choosing $\psi=0$.
	
	Additionally, for any $\phi\in C_\zeta(\Omega)$,
	$$
	\int_{\Omega} (\phi-L)\,d\mu\leq t\sup_{\Omega} (\phi-L)=f(\phi).
	$$
	Therefore,
	$$
	f^\ast(\mu)=\sup_{\phi\in C_\zeta(\Omega)}
	\left(\int_{\Omega}\phi \,d\mu-f(\phi)\right)\leq \int_{\Omega} L\,d\mu.
	$$
Accordingly, we obtain \eqref{hstar}.
\end{proof}

\begin{pro}
\label{gstarp}
Suppose that $\Uu_2\neq\emptyset$.
Let $g$ be as in \eqref{gdef}. Let $\mu\in \Uu$. Then
		$$
	g^*(\mu)=
	\begin{cases}
	0 \quad &\text{ if } \quad -\mu\in\Uu_2\\
	+\infty  &\text{ otherwise.}
	\end{cases}
	$$
\end{pro}
\begin{proof}
Let $\bar \mu\in \Uu_2$; that is, $\mu$ satisfies  
\[
\int_\Omega A^q\varphi \,d\bar  \mu=B\varphi, 
\]
for all $\varphi\in C^2(\Tt^n\times [0,t])$.
Given a measure $\mu$ with  $-\mu\in \Uu_2$, 
decompose $\mu=-\bar \mu+\hat \mu$
with $\hat{\mu}\in\cR(\Omega)$ satisfying 
\begin{equation}\label{eq:Aq}
\int_\Omega A^q\varphi \, d\hat \mu=0 \quad
\text{for all} \  \varphi\in C^2(\Tt^n\times [0,t]). 
\end{equation}
By continuity, we obtain
\[
\int_\Omega \phi \, d\hat \mu=0\quad
\text{for all} \ \phi \in \Cg;
\]
that is,  
\[
\int_\Omega \phi \, d\mu+\int_\Omega \phi \, d\bar \mu=0
\]
for all $\phi \in \Cg$.
%}
Furthermore, if $-\mu\not\in\Uu_2$, 
then there exists $\varphi\in C^2(\T^n\times[0,t])$ such that 
$\int_{\Omega}A^q\varphi\,d(-\mu)\not= B\varphi$. 
By \eqref{eq:Aq}, 
\[
\int_\Omega A^q\varphi \, d(\mu+\hat \mu)=0 \quad
\]
which implies 
$$
\int_\Omega \hat \phi \, d(\mu+\bar \mu)\neq 0 
\quad\text{for some} \ \hat{\phi}\in\cC. 
$$
Thus
$$
g^{*}(\mu)=\sup_{\phi\in\Cg}\left(\int_\Omega \phi \, d\mu+\int_\Omega \phi \, d\bar \mu\right)=
\begin{cases}
0\quad&\text{ if }-\mu\in\Uu_2\\
+\infty&\text{ otherwise}.
\end{cases}
$$
\end{proof}

\begin{teo}
\label{t1}
Suppose that Assumptions \ref{A1}-\ref{A3} hold. Suppose that $\Uu_2\neq\emptyset$.
Then,	
\[
\max_{\mu\in \Uu_1\cap \Uu_2}
\left(-\int_{\Omega} L\, d\mu\right)
=\inf_{\varphi\in C^2(\Tt^n\times [0,t])}  
\left[
-B\varphi+
t\sup_{(x,s)\in\T^n\times[0,t]} \left( \varphi_t -a(x) \Del \varphi + H(x,D\varphi) \right)
\right].
\]	
\end{teo}
\begin{proof}
Let $f$ and $g$ be as in \eqref{hdef} and \eqref{gdef}. 
We have that $f, g :C_\zeta(\Omega)\rightarrow (-\infty,+\infty]$ 
are convex functions.
The set 
$g<+\infty$  is non-empty, and, in this set, 
$f$ is a continuous function. Therefore, 
by the Fenchel-Legendre-Rockafellar theorem \cite{Villanithebook},
\begin{equation*}
%\label{ideq}
\max_{\mu\in \Uu} \left[-f^*(\mu)-g^*(-\mu)\right]=\inf_{\phi\in C_\zeta(\Omega)} \left[f(\phi)+g(\phi)\right]. 
\end{equation*}

By Propositions \ref{fstarp} and \ref{gstarp}, we have 
\[
\max_{\mu\in \Uu} \left[-f^*(\mu)-g^*(-\mu)\right]
=
\max_{\mu\in \Uu_1\cap \Uu_2}
\left(-\int_{\Omega} L\, d\mu\right). 
\]
%Because by the convexity in Assumption \ref{A2}, we have
%\[
%
%\]
%
%
%Since $D_pH(x,\cdot), D_qL(x,\cdot):\R^n\to\R^n$ are $C^1$-diffeomorphism for all $x\in\T^n$, 
Therefore
\begin{align*}
&
\inf_{\phi\in C_\zeta(\Omega)} \left[f(\phi)+g(\phi)\right]
=
\inf_{\phi\in C_\zeta(\Omega)} 
\left[t\sup_{\Omega}(\phi-L)-\int_{\Omega}\phi\,d\overline{\mu}\right]\\ 
=&\,
\inf_{\varphi\in C_\zeta(\Omega)} 
\left[-B\varphi+t\sup_{\Omega}(A^{q}\varphi-L)\right]\\
=&\,
\inf_{\varphi\in C_\zeta(\Omega)} 
\left[-B\varphi+t\sup_{\Tt^n\times[0,t]}
\Big(\varphi_t -a(x)\Del \varphi +\sup_{q\in \Rr^n}\big(q\cdot D\varphi-L(x,q)\big)\Big)\right]\\
=&\,  
\inf_{\varphi\in C_\zeta(\Omega)} 
\left[-B\varphi+t\sup_{(x,s)\in\T^n\times[0,t]}(\varphi_t-a(x)\Del \varphi +H(x,D\varphi))\right],  
\end{align*}
which implies the conclusion. 
\end{proof}

\subsection{Convex duality for $h_t$}

We begin by applying Theorem \ref{t1} to obtain a preliminary dual formula 
for $h_t$. 

%{\color{blue}
%Hung: I think in proposition 5.6 we need for this $B\varphi$ that $\Uu_2$ is non empty. 
%For this we need probably to assume $\cH(\nu_0, \nu_1, 0, t)\neq \emptyset$. 
%}

\begin{pro}
Consider the setting of Problem \ref{P1}.
Suppose that Assumptions \ref{A1}-\ref{A3} hold. 	
Fix $\nu_0, \nu_1\in \cP(\T^n)$, $t>0$, and assume that  $\cH(\nu_0, \nu_1; 0, t)\neq \emptyset$. 
Then,
\begin{align*}
h_t(\nu_0,\nu_1)=
-\inf_{\varphi\in C^2(\Tt^n\times[0,t] )}  \sup_{(x,s)\in\T^n\times[0,t]} 
\Big[&
\int_{\Tt^n} \varphi(x,0)\,d\nu_0-\int_{\Tt^n} \varphi(x,t)\,d\nu_1+\\
&
t\left( \varphi_t(x,s) -a(x)\Del \varphi(x,s) + H(x,D\varphi(x,s)) \right)
\Big].
\end{align*}
	
\end{pro}
\begin{proof}
Given $\nu_0$ and $\nu_1$, 
we set
\begin{equation*}
%\label{Bformula}
B\varphi=\int_{\Tt^n}\varphi(x,t)\,d\nu_1-\int_{\Tt^n}\varphi(x,0)\,d\nu_0.
\end{equation*}
Because  $\cH(\nu_0, \nu_1; 0, t)\neq \emptyset$, we have $\Uu_2\neq\emptyset$.
Then
\[
h_t(\nu_0,\nu_1)=
\inf_{\mu\in \Uu_1\cap \Uu_2}
\int_{\Omega} L \, d\mu, 
\]
where $\Uu_1$ and  $\Uu_2$ are defined in Section \ref{subsec:dual}. The identity in the claim follows from
Theorem \ref{t1}. 
\end{proof}

As stated in Theorem \ref{thm:duality-h},
the preceding result can be further refined; it is enough to use as test functions
viscosity solutions to the Hamilton--Jacobi--Bellman equation in \eqref{eq:1}. This theorem
is proved as follows.

\begin{proof}[Proof of Theorem \ref{thm:duality-h}]
We denote by $m_t(\nu_0,\nu_1)$ the right-hand side of \eqref{eq:dual-ht}. 

Let $w$ be a viscosity solution to  \eqref{eq:1}  and consider 
the approximation $w^{\alpha}$ defined in \eqref{func:w-al}.
Then,
\begin{align*}
&\inf_{\varphi\in C^2(\Tt^n\times [0,t])}  \sup_{(x,s)} 
\Big[
\int_{\Tt^n} \varphi(x,0)\,d\nu_0-\int_{\Tt^n} \varphi(x,s)\,d\nu_1\\
&\qquad \qquad\qquad \qquad\quad+
t\left( \varphi_t(x,s) -a(x)\Del \varphi(x,s) + H(x,D\varphi(x,s)) \right)
\Big]
\\ \leq \, &
\sup_{(x,s)} 
\Big[
\int_{\Tt^n} w^{\alpha}(x,0)\,d\nu_0-\int_{\Tt^n} w^{\alpha}( x,t) \,d\nu_1+
t\left( w^{\alpha}_t -a(x)\Del w^{\alpha} + H(x,Dw^{\alpha}) \right)
\Big].
\\ \leq \, &
\left[\int_{\Tt^n} w^{\alpha}(x,0)\,d\nu_0-\int_{\Tt^n} w^{\alpha}( x,t) \,d\nu_1\right] + C \al^{1/2}t.
\end{align*}
Thus, letting $\al \to 0$, we conclude that 
\begin{align*}
&\inf_{\varphi\in C^2(\Tt^n\times [0,t])}  \sup_{(x,s)} 
\Big[
\int_{\Tt^n} \varphi(x,0)\,d\nu_0-\int_{\Tt^n} \varphi(x,s)\,d\nu_1\\
&\qquad \qquad\qquad \qquad\qquad+
t\left( \varphi_t(x,s) -a(x) \Del \varphi(x,s) + H(x,D\varphi(x,s)) \right)
\Big]
\\
\leq \, &
\int_{\Tt^n} w(x,0)\,d\nu_0-\int_{\Tt^n} w( x,t)\,d\nu_1, 
\end{align*}
which implies $h_t(\nu_0,\nu_1)\ge m_t(\nu_0,\nu_1)$.

Now take  $\varphi\in C^2(\Tt^n\times [0,t])$. By subtracting $C t$ to $\varphi$ with $C>0$ large enough,
we can assume that 
\[
\varphi_t -a(x)\Del \varphi + H(x,D\varphi)\leq 0;
\]
that is, $\varphi$ is a classical subsolution of \eqref{eq:1}. 
Next, let $w^\varphi$ be the solution of the Hamilton--Jacobi--Bellman equation \eqref{eq:1} with initial data
$w^\varphi(x,0)=\varphi(x,0)$. Then $\varphi(x,t)\leq w^\varphi(x,t)$.
Accordingly,
\[
\int_{\Tt^n} w^\varphi(x,0) \, d\nu_0-\int_{\Tt^n} w^\varphi(x,t)\,d\nu_1\leq \int_{\Tt^n} \varphi(x,0)\,d\nu_0-\int_{\Tt^n} \varphi(x,t)\,d\nu_1.
\]
Therefore,
\begin{align*}
&\inf_{\varphi\in C^2(\Tt^n\times [0,t])}  \sup_{(x,s)\in\T^n\times[0,t]} 
\Big[
\int_{\Tt^n} \varphi(x,0) \, d\nu_0-\int_{\Tt^n} \varphi(x,t)\, d\nu_1+
\\
&\qquad \qquad \qquad\qquad\qquad+
t\left( \varphi_t -a(x)\Del \varphi + H(x,D\varphi) \right)
\Big]
\\ \geq \, &
\inf_{\varphi\in C^2( \Tt^n\times [0,t])} 
\left(\int_{\Tt^n} w^\varphi(x,0)\, d\nu_0-\int_{\Tt^n} w^\varphi(x,t)\, d\nu_1 \right)
\\ \ge \, & 
\inf_{w \in \cE} \left( \int_{\Tt^n} w(x,0)\, d\nu_0-\int_{\Tt^n} w(x,t)\, d\nu_1 \right),
\end{align*}
which implies $h_t(\nu_0,\nu_1)\le m_t(\nu_0,\nu_1)$. 
\end{proof}

\subsection{Convex duality for $d$ on $\cM$}

Now, we examine the function $d$ in light of duality methods and
prove Theorem \ref{thm:duality-d}. 

\begin{pro}
	\label{p58}
Consider the setting of Problem \ref{P1} and
suppose that Assumptions \ref{A1}-\ref{A3} hold.
Let $\nu_0, \nu_1\in \cP(\Tt^n)$.
Then	
\begin{equation}
\label{lb}
d(\nu_0,\nu_1)\geq \sup_{w \in \Ss}\left[
\int_{\T^n}w \,d\nu_1
-
\int_{\T^n}w \,d\nu_0\right],
\end{equation}
where $\Ss$ is the set of all Lipschitz viscosity solutions $w \in C(\T^n)$ of \eqref{ee} defined by \eqref{def:Ss}. 
\end{pro}
\begin{proof}
Recall that 
\[
d(\nu_0,\nu_1)=\liminf_{t\to +\infty} h_t(\nu_0,\nu_1).
\]
Let $w \in C(\Tt^n)$ solve \eqref{ee}.  
Then, $w \in \Lip(\T^n)$.
Because of our 
convention that $c=0$,  $w$ solves  \eqref{eq:1}.
Therefore
\[
h_t(\nu_0,\nu_1)
\geq \left[
\int_{\T^n}w \,d\nu_1
-
\int_{\T^n}w \,d\nu_0\right].
\]
Consequently, we obtain \eqref{lb}.
\end{proof}

\begin{lem}
\label{monlem}
Consider the setting of Problem \ref{P1} and
suppose that Assumptions \ref{A1}-\ref{A3} hold.  Let $\nu_0\in \Mm$. 	
Given $\varphi\in C(\Tt^n)$, let $w^\varphi$ be the solution of  \eqref{eq:1} with initial condition
$w^\varphi(x,0)=\varphi(x)$. Then, the map
\[
t \mapsto \int_{\T^n} w^\varphi(x,t)\,d\nu_0
\]
is non-increasing. 
\end{lem}
\begin{proof}

Let $\mu_0$ be a Mather measure such that $\nu_0 = \proj_{\T^n} \mu_0$, and $\eta$ be a measure such that
\[
d\eta(x,q,s) = d\mu_0(x,q) \qquad \text{ for all } s\geq 0.
\]
By the representation formula for $w_f$, we have
\[
\int_{\T^n} w^\varphi(x,t)\,d\nu_0 \leq \int_{\T^n} \varphi\,d\nu_0 + \int_{\T^n \times \R^n \times [0,t]} L(x,q) \,d\eta(x,q,s)=\int_{\T^n} \varphi\,d\nu_0.
\]
We hence see that for $t>0$
\[
 \int_{\T^n} w^\varphi(x,t)\,d\nu_0\leq \int_{\T^n} w^\varphi(x,0)\,d\nu_0.
\]
Thus, by the semigroup property of 
the solution operator of the Hamilton--Jacobi--Bellman equation, we get the 
claim in the statement. 
%
%and in particular,
%\[
%\int_{\T^n} w \,d\nu_0 \leq \int_{\T^n} f\,d\nu_0.
%\]
%Same conclusion holds for $\nu_1$ as it is also a projected Mather measure.
\end{proof}

%\begin{lem}
%\label{CP}
%	Compacity lemma: for any $t>0$ the map $\phi(x,0)\mapsto \phi(x,t)$ is compact in $C(\Tt^d)$.
%\end{lem}
%\begin{proof}
%	content...
%\end{proof}

%{\color{red} this lemma should be  stated in the introduction?}

%\begin{lem}
%\label{conv}
%Consider the setting of Problem \ref{P1} and
%suppose that Assumptions  \ref{A2}-\ref{A4} hold. 	
%Given $\ep>0$, there exists $T$ such that
%for any initial data $u_0\in C^2(\Tt^n)$
%the corresponding solution, $u$, to Problem \ref{P1}
%satisfies
%\[
%|u(x, t) -u_\infty|<\ep, 
%\]
%for all $t>T$, 
%where
%\[
%u_\infty=\lim_{t\to \infty} u(x, t).
%\]
%\end{lem}

\begin{pro}
	\label{p510}
Consider the setting of Problem \ref{P1} and
suppose that Assumptions  \ref{A1}-\ref{A4} hold.
Let $\nu_0, \nu_1\in \Mm$ and suppose that $d(\nu_0, \nu_1)<+\infty$.
Then 
	\[
	d(\nu_0,\nu_1)\leq \sup_{w \in \Ss}\left[
	\int_{\T^n}w \,d\nu_1
	-
	\int_{\T^n}w \,d\nu_0\right].
	\]	
\end{pro}

\begin{proof}
Fix $\nu_0, \nu_1\in\cP(\T^n)$. 

Fix $\ep>0$ and let $T$ be as in the statement of Proposition \ref{conv}. Let $\{t_j\}$
be a sequence with 
$t_j\to\infty$ and
$h_{t_j}(\nu_0,\nu_1)\to d(\nu_0,\nu_1)$ as $j\to\infty$. 
Select $t_j>T$ such 
that 
\[
|d(\nu_0, \nu_1)- h_{t_j}(\nu_0, \nu_1)|\leq \ep. 
\]
Next, by Theorem \ref{thm:duality-h}, for every $\ep>0$ there exists a Lipschitz viscosity solution $u$ to  \eqref{eq:1} such that
\[
h_{t_j}(\nu_0, \nu_1)\leq 
\int_{\T^n}u(\cdot,t_j)\,d\nu_1
-
\int_{\T^n}u(\cdot,0)\,d\nu_0+\ep
\]
Let $u_\infty(x)=\lim_{t\to \infty} u(x,t)$ for $x\in \Tt^n$. According to Lemma \ref{monlem},
\[
\int_{\T^n}u(\cdot,0)\,d\nu_0\geq \int_{\T^n}u_\infty \,d\nu_0.
\]
Moreover, due to the convergence result in Proposition \ref{conv}, we have
\[
\int_{\T^n}u(\cdot,t)\,d\nu_1\leq \int_{\T^n}u_\infty \,d\nu_1+\ep
\quad\text{for all} \ t>T. 
\]
Therefore,
\[
d(\nu_0, \nu_1)\leq \int_{\T^n}u_\infty \,d\nu_1-\int_{\T^n}u_\infty \,d\nu_0
+3\ep\le m(\nu_0,\nu_1)+3\ep. 
\]
Because $\ep$ is arbitrary and $u_\infty$ solves \eqref{ee}, we obtain the desired inequality.
\end{proof}

\begin{proof}[Proof of Theorem \ref{thm:duality-d}]
Theorem \ref{thm:duality-d} is a straightforward result of Propositions \ref{p58} and \ref{p510}.
\end{proof}

\section{Proof of Theorem \ref{thm:A4}}\label{sec:A4}

In this section, we prove Theorem \ref{thm:A4} in case that $m>2$. 
If $a(x)\equiv 0$, the compactness of $S_1$ 
follows from the coercivity of the Hamiltonian. 

\begin{proof}[Proof of Theorem \ref{thm:A4}]
It is enough to show the compactness of $S_1$; that is, 
to prove that $\{S_1v\}_{v\in X_{0}}$ is uniformly bounded and equicontinuous. 
The main difficulty here is that $X_{0}=\{u_0 \in C(\T^n)\,:\,\min_{\T^n} u_0=0\}$ is 
not bounded in $C(\T^n)$.

Take $u_0 \in X_0$.
Without loss of generality, assume that
\[
u_0(0)=\min_{\T^n} u_0 =0.
\]
It is clear that for any $C>0$, $\phi(x,t)=-Ct$ is a subsolution to \eqref{eq:1}. Therefore,
\[
u(x,t) \geq -Ct \quad \text{ in } \T^n \times [0,\infty).
\]
We need to construct a supersolution to bound $u$ from above. For this,  the superquadratic growth of $H$ is crucial. Here, 
we borrow some ideas from \cite{cardaliaguetsilvestre}.
For any $C>0$ and $\lambda>0$ to be chosen later, let
\[
\psi(x,t) = \lambda t^{-\frac{1}{m-1}} \left(|x|^2 + \lambda t\right)^{\frac{m'}{2}} + Ct \quad \text{ for all } (x,t) \in \R^n \times (0,\infty).
\]
Note that $m' = \frac{m}{m-1} \in (1,2)$ because $m>2$. 
We compute that
\begin{align*}
\psi_t&=\lambda t^{-m'} \left(|x|^2 + \lambda t\right)^{\frac{m'}{2}} \left ( \frac{m'}{2} \frac{\lam t}{|x|^2+\lam t}  - \frac{1}{m-1}\right) + C,\\
\psi_{x_i}&= \lambda t^{-\frac{1}{m-1}} \left(|x|^2 + \lambda t\right)^{\frac{m'}{2}-1} m' x_i,\\
\psi_{x_i x_j}&= \lambda t^{-\frac{1}{m-1}} \left(|x|^2 + \lambda t\right)^{\frac{m'}{2}-1} m'\left( \delta_{ij}  + \frac{(m'-2) x_i x_j}{|x|^2+\lam t}\right).
\end{align*}
Therefore, for $(x,t) \in \R^n \times (0,\infty)$, noting that $mm'=m+m'$, 
\[
|D\psi(x,t)|^m = (m')^m \lambda^m t^{-m'} \left(|x|^2 + \lambda t\right)^{\frac{m'}{2}} \left( \frac{|x|^2}{|x|^2+\lambda t} \right)^{\frac{m}{2}},
\]
and
\[
a(x)\Del \psi(x,t) \leq C \lambda t^{-m'} \left(|x|^2 + \lambda t\right)^{\frac{m'}{2}} \frac{t}{|x|^2+\lam t}.
\]
Combining the above computations, for $(x,t) \in \R^n \times (0,\infty)$,
we gather 
\begin{align*}
&\psi_t + H(x,D\psi) - a(x)\Del \psi \geq \psi_t + \frac{|D\psi|^m}{C}-C -  a(x)\Del \psi \ \\
\geq \ &\lambda t^{-m'} \left(|x|^2 + \lambda t\right)^{\frac{m'}{2}} \left ( \left( \frac{\lam m'}{2} -C \right) \frac{ t}{|x|^2+\lam t}  - \frac{1}{m-1} +\frac{\lambda^{m-1}}{C}  \left( \frac{|x|^2}{|x|^2+\lambda t} \right)^{\frac{m}{2}}\right).
\end{align*}
Because $m>2$,
\[
\frac{m'}{2} = \frac{m}{2} \frac{1}{m-1} > \frac{1}{m-1}.
\]
Accordingly, let $\theta= \frac{1}{2}\left(\frac{m'}{2} -\frac{1}{m-1}\right)>0$.
Then, for $\lam>0$ sufficiently large,
\[
 \frac{\lam m'}{2} -C \geq \lambda \left(\frac{1}{m-1} +\theta\right).
\]
Let $z =\frac{|x|^2}{|x|^2+\lambda t} \in [0,1]$.
Then,
\begin{align*}
 &\left( \frac{\lam m'}{2} -C \right) \frac{ t}{|x|^2+\lam t}  - \frac{1}{m-1} +\frac{ \lambda^{m-1}}{C}  \left( \frac{|x|^2}{|x|^2+\lambda t} \right)^{\frac{m}{2}}\\
 \geq \ &  \left(\frac{1}{m-1} +\theta\right) ( 1-z)  - \frac{1}{m-1} +\frac{\lambda^{m-1}}{C} z^{\frac{m}{2}}\\
= \ & \frac{\lambda^{m-1}}{C} z^{\frac{m}{2}} + \theta -   \left(\frac{1}{m-1} +\theta\right)z > 0
\end{align*}
for $\lam>0$ sufficiently large depending only on $C, n, m$.

Thus, $\psi$ is a classical strict supersolution to \eqref{eq:1} in $ \R^n \times (0,\infty)$, and $\psi(0,0)=0$, $\psi(x,0)=+\infty$ for $x\neq 0$.
In particular, $\psi(x,0) \geq u_0(x)$ for all $x\in \R^n$.
The viscosity subsolution test implies that $u-\psi$ cannot have a maximum point at $(x_0,t_0) \in \R^n \times (0,\infty)$.
Therefore,  $u \leq \psi$ in $ \R^n \times (0,\infty)$.
In particular, for $x \in [0,1]^n$ and $t=\frac{1}{2}$,
\begin{equation*}
-\frac{C}{2} \leq u\left(x,\frac{1}{2}\right)\leq \lambda 2^{\frac{1}{m-1}} \left(|x|^2 + \frac{\lambda}{2} \right)^{\frac{m'}{2}} + \frac{C}{2} \leq C_0, 
\end{equation*}
where $C_0$ depends only on $C, n, m$.
Thus, we see that, for any $u_0 \in C(\T^n)$ such that $\min_{\T^n} u_0 =0$,
\begin{equation}\label{eq: u bounded 1-2}
\|S_{\frac{1}{2}} u_0\|_{L^\infty(\T^n)} \leq C_0.
\end{equation}
We now use \eqref{eq: u bounded 1-2} together with \cite[Theorem 1.2]{cardaliaguetsilvestre} to conclude that $S_1 u_0 \in C^{0,\alpha}(\T^n)$, where $\alpha \in (0,1)$ and $\|S_1 u_0\|_{C^{0,\alpha}(\T^n)}$ depend only on $C, C_0, m,n$.
The proof is hence complete.
\end{proof}

For further H\"older regularity results with respect to superquadratic Hamiltonians, we refer the readers to \cite{A-T, CapuzzoDolcetta2010, cardaliaguetcannarsa} and the references therein.

%%%%% APPPENDIX %%%%
\appendix 
\section{Approximations of solutions to Problem \ref{P1}} \label{Appen:A}

Let $\theta \in C_c^\infty(\R^{n},[0,\infty))$ and $\rho \in C_c^\infty(\R,[0,\infty))$
be symmetric standard mollifiers; that is, 
$\supp \theta\subset\ol{B}(0,1)\subset\R^{n}$, $\supp \rho\subset\ol{B}(0,1)\subset\R$, $\theta(x)=\theta(-x)$, $\rho(s)=\rho(-s)$,
and $\|\theta\|_{L^{1}(\R^{n})}=\|\rho\|_{L^{1}(\R)}=1$. 
For $\alpha>0$, set $\theta^\alpha(x):=\alpha^{-n} \theta(\alpha^{-1}x)$ for $x\in \R^{n}$, 
and $\rho^{\alpha}(t):=\alpha^{-1} \rho(\alpha^{-1}t)$ for $t\in \R$. 
For $w \in C(\T^n \times [0,\infty))$, let $w^{\alpha}\in C^\infty(\T^n\times[\al,\infty))$ be 
\begin{equation}\label{func:w-al}
w^{\al}(x,t)
:=
\int_{0}^{\infty} \rho^\al(s)\int_{\Rr^n} \theta^\al(y)w(x-y,t-s)\,dy ds \quad \text{ for $(x,t)\in\T^n\times[\al,\infty)$.}
\end{equation}

\begin{pro}\label{prop:approx A}
Let $w$ be a Lipschitz solution to \eqref{eq:1}.
For $0<\al<1$, let $w^\al$ be as in \eqref{func:w-al}.
Then, there exists $C>0$ depending on $H, a$, and the Lipschitz constant of $w$ such that
\begin{equation}\label{eq:alpha}
w^\al_t - a(x)\Del w^\al + H(x,Dw^\al) \leq C \al^{1/2} \quad \text{ on } \T^n \times [\al, \infty).
\end{equation}
\end{pro}

We give only a brief outline of the proof. For a detailed proof, see \cite{MR3581314, LMT}.

\begin{proof}[Proof outline]
To obtain the inequality in \eqref{eq:alpha}, we seek to rewrite its left-hand side in terms of 
the convolution of 
the left-hand side in \eqref{eq:1} with $ \rho^\al  \theta^\al$. 
We will handle each of the terms separately. 
The first term, $w^\al_t$ is trivial.  For the last term, $H(x,Dw^\al)$, we observe that  
Jensen's inequality gives, for $(x,t)\in\T^n\times[\al,\infty)$,
\begin{align*}
H(x,Dw^{\al}(x,t)) &= H\left(x, \int_{0}^{\infty} \rho^\al(s)\int_{\T^n} \theta^\al(y)Dw(x-y,t-s)\,dy ds\right)\\
&\leq \int_{0}^{\infty} \rho^\al(s)\int_{\T^n} \theta^\al(y)H(x,Dw(x-y,t-s))\,dy ds\\
&\leq \int_{0}^{\infty} \rho^\al(s)\int_{\T^n} \theta^\al(y)H(x-y,Dw(x-y,t-s))\,dy ds + C \alpha.
\end{align*}
Thus, the term $H(x,Dw^{\al}(x,t))$ is controlled by the corresponding term in 
\eqref{eq:1} convolved with $ \rho^\al  \theta^\al$ and an error term bounded by $C \alpha$.

The second term, $ a(x)\Del w^\al $, is where the main difficulty of the estimate lies. 
Because $w$ is Lipschitz, using equation \eqref{eq:1}, we have
\begin{equation}\label{eq:distribution}
-C \leq a(x) \Del w \leq C \quad \text{ on } \T^n \times [0,\infty)
\end{equation}
in viscosity sense.

Because of the simple structure of $a$, we see further that $\|a\Del w\|_{L^{\infty}(\T^n \times [0,\infty))} \leq C$, and $w$ is a subsolution to \eqref{eq:1} and \eqref{eq:distribution} in the distributional sense.
We need to control the commutation term,
\begin{align*}
&\int_{0}^{\infty} \rho^\al(s)\int_{\T^n} \theta^\al(y)a(x-y)\Del w(x-y,t-s)\,dy ds - a(x) w^{\al}(x,t)\\
 =\, & \int_{0}^{\infty} \rho^\al(s) \left(\int_{\T^n} \theta^\al(y)(a(x-y)-a(x))\Del w(x-y,t-s)\,dy \right)\, ds \\
 =\, & \int_{0}^{\infty} \rho^\al(s) R^\al(x,t-s)\,ds,
\end{align*}
where
\[
R^\al(x,t-s) = \int_{\T^n} \theta^\al(y)(a(x-y)-a(x))\Del w(x-y,t-s)\,dy.
\]
To complete the proof, we show that $|R^\al(x,t)| \leq C \alpha^{1/2}$  for all $(x,t) \in \T^n \times [0,\infty)$.
We consider two cases
\[
\text{(i) } \min_{y\in \ol{B(x,\al)}} a(y) \leq \al, \quad \text{ or } \quad  \text{(ii) } \min_{y\in \ol{B(x,\al)}} a(y) > \al.
\]

In case (i), there exists $\bar x \in \ol{B(x,\al)}$ such that $a(\bar x) \leq \al$. 
Then, there exists a constant $C>0$ such that,
\[
|Da(\bar x)| \leq C a(\bar x)^{1/2} \leq C \al^{1/2}.
\]
See \cite[Lemma 2.6]{CGMT} for example.
For any $z\in \ol{B(x,\al)}$,
\begin{equation*}
|Da(z)| \leq |Da(z)-Da(\bar x)|+ |Da(\bar x)| \leq C \al + C \al^{1/2} \leq C \al^{1/2}. 
\end{equation*}
Moreover, by using Taylor's expansion, 
\begin{align*}
& |a(z)-a(x)| \leq  
|a(z)-a(\bar x)|+ |a(x)-a(\bar x)|\\
\leq \ & |Da(\bar x)| (|z-\bar x|+|x-\bar x|)
+ C(|z-\bar x|^2+|x-\bar x|^2) \leq C \al^{3/2}+ C\al^2 \leq C \al^{3/2}.
\end{align*}
We use the two above inequalities to control  $R^\al(x,t)$ as 
\begin{align*}
&|R^\al(x,t)| = \left| \int_{\R^n} (a(x-y)-a(x)) \Del w(x-y,t) \theta^\al(y)\,dy\right|\\
=\, &\left| \int_{\Rr^n} Dw(x-y,t)\cdot Da(x-y) \theta^\al(y)\,dy - \int_{\Rr^n} Dw(x-y,t)\cdot D\theta^\al(y) (a(x-y)-a(x))\,dy\right|\\
\leq\, & C \int_{\Rr^n} \left(\al^{1/2} \theta^\al(y)+ \al^{3/2} |D\theta^\al(y)|\right)\,dy \leq C \al^{1/2}.
\end{align*}

Now, we  consider case (ii); that is, $\min_{\ol{B(x,\al)}}\, a>\al$. A direct computation shows that
\begin{align*}
&|R^\al(x,t)|\le \int_{\Rr^n} \left| (a(x-y)-a(x))\right|\cdot \left| \Del w(x-y,t)\right| \theta^\al(y)\,dy\\
\leq \ &C \int_{\Rr^n} \frac{|a(x-y)-a(x)|}{a(x-y)} \theta^\al(y)\,dy
\leq C \int_{\Rr^n} \frac{|Da(x-y)| \cdot |y|}{a(x-y)} \theta^\al(y)\,dy+C\al\\
\leq  \ &C \int_{\Rr^n} \frac{|y|}{a(x-y)^{1/2}} \theta^\al(y)\,dy+C\al
\leq C \int_{\Rr^n} \frac{|y|}{\al^{1/2}} \theta^\al(y)\,dy+C\al \leq C \al^{1/2}. 
\end{align*}
Combining these estimates yields the conclusion.
\end{proof}

\section{The general diffusion matrix case} \label{Appen:B}

Now, we consider the case of a  general  diffusion matrix $A$.
Let $\bS^n$ be the space of $n \times n$ real symmetric matrices.
Let $A:\T^n\to\bS^n$ be a nonnegative definite diffusion matrix; that is, $ \xi^T A(x)\xi\geq 0$ for all $\xi\in \Rr^n$ and $x\in \T^n$. 
Assume further that $A\in C^2(\T^n, \bS^n)$.
We suppose that Assumptions \ref{A1}--\ref{A3} always hold in this section and replace \eqref{eq:1} in Problem \ref{P1} by
 the following general Hamilton--Jacobi--Bellman equation
\begin{equation} \label{eq:1B}
u_t -\tr (A(x)D^2 u) + H(x,Du) = 0 \quad \text{ in }  \T^n \times (0,\infty), 
\end{equation}
where 
$Du$ and $D^2 u$, respectively, denote the spatial gradient and Hessian of $u$, 
and with initial data
\begin{equation*}
%\label{eq:1bcB}
u(x,0) = u_0(x) \quad \text{ on } \T^n.
\end{equation*}
We now extend the results in Theorems \ref{thm:rep}--\ref{thm:profile} for this setting.
While the statements are similar, there are several technical points that must be addressed.  The main difficulty is that the analog of the approximation
result in Proposition \ref{prop:approx A} is substantially more involved. 
This approximation result is examined in Proposition \ref{prop:approx}
and requires the approximate equation to be uniformly parabolic. Thus, 
further approximation arguments are needed in various places, including 
in the definition of holonomic measures.

\subsection{Representation formulas for the general case of diffusion matrices}
For $\nu_0, \nu_1\in\cP(\T^n)$, $\eta>0$, 
let $\cH^{\eta}(\nu_0,\nu_1;t_0,t_1)$ be the set of all 
	$\gamma\in \cR^+(\Tt^n\times \Rr^n\times [t_0,t_1])$
	 satisfying
\begin{equation*}
\label{zetaintB}
\int_{\T^n \times \R^n\times[t_0,t_1]} |q|^\zeta \,d\gam(z,q,s)<\infty
\end{equation*}
and	 	 
	\begin{align*}
	&\int_{\T^n \times \R^n\times[t_0,t_1]} 
	\left(\varphi_t(z,s)-\eta\Delta\varphi(z,s)-\tr(A(z)D^2\varphi(z,s))+q\cdot D\varphi(z,s) \right)
	\,d\gam(z,q,s) \nonumber \\
	=\, &\, 
	\int_{\T^n}\varphi(z,t_1)\,d\nu_1(z)
	-\int_{\T^n}\varphi(z,t_0)\,d\nu_0(z)  
	\label{def:1B}
	\end{align*}
	for all $\varphi\in C^2(\T^n\times[t_0,t_1])$. 
	Moreover, we set 
	\[
	\cH^\eta(\nu_1;t_0,t_1):=\bigcup_{\nu_0\in\cP(\T^n)} \cH^\eta(\nu_0,\nu_1;t_0,t_1). 
	\]
Fix any $\nu_1 \in\cP(\T^n)$. 
It is worth emphasizing that while we do not know that $\cH^\eta(\nu_0,\nu_1;t_0,t_1)\neq \emptyset$ for each $\nu_0\in\cP(\T^n)$, 
the set $\cH^\eta(\nu_1;t_0,t_1)$ is always non-empty as shown in the proof of Lemma \ref{L1B}.

We define $\widetilde \cH(\nu_1;t_0,t_1)$ as follows. 
We say that $\gamma\in\widetilde \cH(\nu_1;t_0,t_1)$ if $\gamma\in\cR^+(\T^n\times\R^n\times[t_0,t_1])$, and there exist $C>0$, a sequence $\{\eta_j\}_{j\in\N} \to0$, $\{\nu_1^{\eta_j}\}_{j \in \N}\subset\cP(\T^n)$, and $\{\gamma^{\eta_j}\}_{j \in \N}\subset\cR^+(\T^n\times\R^n\times[t_0,t_1])$ such that 
\begin{align*}
&\gamma^{\eta_j}\in \cH^{\eta_j}(\nu_1^{\eta_j};t_0,t_1), \ \text{and} \ 
\int_{\T^n \times \R^n\times[t_0,t_1]} |q|^\zeta \,d\gam^{\eta_j}(z,q,s)<C \quad \text{ for each } j\in \N, \\
&\nu_1^{\eta_j}\rightharpoonup \nu_1 \quad\text{ weakly in the sense of measures in} \ 
\cR(\T^n),\\
&\gamma^{\eta_j}\rightharpoonup \gamma \quad\text{ weakly in the sense of measures in} \ 
\cR(\T^n\times\R^n\times[t_0,t_1]) \ \text{as} \ j\to\infty. 
\end{align*}   
We also note that $\widetilde\cH(\nu_1;t_0,t_1)$ is non-empty for any $\nu_1\in\cP(\T^n)$ as stated in Corollary \ref{cor:non-emptyB}. 
For  $\gamma\in \widetilde  \cH(\nu_1;t_0,t_1)$, let $\nu^{\gamma}$ be the unique element in $\cP(\T^n)$ such that
\begin{align*}
	&\int_{\T^n \times \R^n\times[t_0,t_1]} 
	\left(\varphi_t(z,s)-\eta\Delta\varphi(z,s)-\tr(A(z)D^2\varphi(z,s))+q\cdot D\varphi(z,s) \right)
	\,d\gam(z,q,s) \nonumber \\
	=\, &\, 
	\int_{\T^n}\varphi(z,t_1)\,d\nu_1(z)
	-\int_{\T^n}\varphi(z,t_0)\,d\nu^\gam(z)  
	\label{def:1B}
	\end{align*}
	for all $\varphi\in C^2(\T^n\times[t_0,t_1])$.
	Accordingly, $\gamma\in \widetilde \cH(\nu^{\gamma},\nu_1;t_0,t_1)$.

We now use the measures in  $\widetilde \cH(\nu_1;t_0,t_1)$ to obtain a representation formula for solutions of \eqref{eq:1B}, which can be viewed as a generalization of Theorem \ref{thm:rep}.

\begin{teo}\label{thm:rep B}
	Let $u$ solve \eqref{eq:1B}.  
	Suppose that Assumptions  \ref{A1}-\ref{A3} hold.
	Then,
	for any $\nu\in\cP(\T^n)$ and $t>0$, we have
	\begin{equation*}
	\int_{\T^n}u(z,t)\,d\nu(z)=
	\inf_{\gam\in\widetilde \cH(\nu;0,t)}
	\left[
	\int_{\T^n\times\R^n\times[0,t]}L(z,q)\,d\gam(z,q,s)+\int_{\T^n}u_0(z)\,d\nu^\gamma(z)
	\right].
	\end{equation*}
\end{teo}

This theorem is proved in the next subsection. 
%\medskip

The ergodic problem here is
	\begin{equation}
	\label{eeB}
	 -\tr(A(x)D^2 v)+H(x,Dv)=c \quad \text{ in }  \T^n. 
	\end{equation}
	As previously (cf Remark \ref{czero}), 
we add a constant to  $H$, if necessary,  so that  $c=0$.
As in the time-dependent case, we begin by defining the sets of approximated stationary generalized holonomic measures.
For each $\eta>0$, we denote by
\begin{multline*}
\cH^\eta:=\Big\{\mu\in\cP(\T^n\times\R^n) \,:\,
\int_{\T^n \times \R^n} |q|^\zeta \,d\mu(z,q)<\infty,\\
\int_{\T^n\times\R^n}\left(-\eta\Delta\varphi(z)-\tr(A(z)D^2\varphi(z))+q\cdot D\varphi(z)\right)\,d\mu(z,q)=0 
\quad\text{ for all} \ \varphi\in C^2(\T^n)\Big\} . 
\end{multline*}

In a similar manner to  the time-dependent case, 
 $\widetilde \cH$  is the  set of all
%We say that $\mu\in \widetilde \cH$ if 
$\mu\in\cP(\T^n \times \R^n)$ for which there exist $C>0$, $\{\eta_j\}_{j\in \N} \to0$, and $\{\mu^{\eta_j}\}_{j\in\N}\subset\cP(\T^n \times \R^n)$ such that 
\begin{align*}
&\mu^{\eta_j}\in\cH^{\eta_j}, \text{ and } \int_{\T^n \times \R^n} |q|^\zeta \,d\mu^{\eta_j}(z,q)< C \quad \text{ for all } j \in \N,\\
&\mu^{\eta_j}\rightharpoonup \mu \quad\text{ weakly in the sense of measures in} \ 
\cP(\T^n \times \R^n) \ \text{as} \ j\to\infty. 
\end{align*}
We consider the variational problem 
\begin{equation}\label{def:Mather1B}
\inf_{\mu\in \widetilde\cH}
\int_{\Tt^n\times \Rr^n}L(x,q)\,d\mu(x,q). 
\end{equation}
A {\em generalized Mather measure} is a
solution of the minimization problem in \eqref{def:Mather1B} and
$\widetilde{\cM}$ is the set of all generalized Mather measures. 
Moreover, $\cM$ is the set of all generalized projected Mather measures; that is,  the projections to $\T^n$ of generalized Mather measures. 

\begin{proposition}\label{prop:erg}
Suppose that Assumptions \ref{A1}-\ref{A3} hold.
Assume that the ergodic constant $c$ for \eqref{eeB} is $0$.  
We have 
\begin{equation*}%\label{prop:erg-const}
\inf_{\mu\in\widetilde \cH} \int_{\Tt^n\times \Rr^n}
L(x,q)\,d\mu(x,q)=0.  
\end{equation*}
\end{proposition} 

This proposition is proved at the end of the paper. 

Finally, we have the following representation result, which is a generalized version of Theorem \ref{thm:profile}.

\begin{teo}\label{thm:profileB}
	Suppose that Assumptions  \ref{A1}-\ref{A3} hold.
	Let $u$ solve \eqref{eq:1B} and
	$u_\infty$ be as in \eqref{uinfdef}.  
	Then, for any $\nu\in\cM$, we have
	\begin{equation*}
	\label{maverB}
	\int_{\T^n} u_\infty(z)
	 \,d\nu(z) = \inf_{\nu_0\in\cP(\T^n)} 
	\left[d(\nu_0,\nu) + \int_{\T^n} u_0(z)\,d\nu_0(z)  \right], 
	\end{equation*}
where for $\nu_0, \nu_1 \in \cP(\T^n)$, $d$ is the  generalized Ma\~n\'e critical potential connecting $\nu_0$ to $\nu_1$ given by 
\begin{align*}
d(\nu_0,\nu_1):=
\inf_{\substack{\gam\in \widetilde \cH(\nu_0,\nu_1;0,t)\\ t>0}}  
\int_{\T^n \times \R^n \times [0,t]} L(x,q) \, d\gam(x,q,s).
\label{def:dB}
\end{align*}
\end{teo}

The proof of the preceding theorem is similar to the one of Theorem \ref{thm:profile}, so we
omit it here.

%\begin{remark}
%The constructions of $\widetilde \cH(\nu_0,\nu_1;0,t)$ and $\widetilde \cH$ for a general diffusion matrix $A$ are more complex and less intuitive than those of  $ \cH(\nu_0,\nu_1;0,t)$ and $ \cH$ for the case where $A(x)=a(x)I_n$.
%\end{remark}

\subsection{Proof of Theorem \ref{thm:rep B} }

We begin by  proving Theorem \ref{thm:rep B}.
\begin{lem}
\label{VLB}
Suppose that Assumptions \ref{A1}-\ref{A3} hold.
Let $u$ solve \eqref{eq:1B}. 
Then, for $\nu_1 \in \cP(\T^n)$ and $t>0$,
\begin{equation*}\label{ineq:VLB}
\int_{\Tt^n}u(z,t)d\nu_1(z) \leq
\inf_{\gamma\in \widetilde \cH(\nu_1; 0,t)}
 \left[
 \int_{\T^n \times \R^n\times[0,t]} 
 L(z,q)\,d\gam(z,q,s)+
\int_{\T^n}u(z,0)\,d\nu^\gamma(z)
\right]. 
\end{equation*}
\end{lem}

\begin{proof}
Note that $u$ is globally Lipschitz continuous on $\T^n\times[0,\infty)$ 
(see \cite[Proposition 3.5]{A-T}, \cite[Proposition 4.15]{LMT} for instance)
under Assumptions \ref{A1}--\ref{A3}.

For $\alpha, \ep, \delta>0$,  let $u^{\al,\ep,\del}$ be the function given by \eqref{func:al-ep-del} in Section \ref{subsec:approx} below, and set 
\[
\tilde{u}(x,t):=u^{\al,\ep,\del}(x,t+\alpha) \quad \text{ for all $(x,t)\in\T^n\times[0,\infty)$. }
\]
We notice that $\tilde u\in C^2(\T^n\times[0,\infty))$ and it is an
approximate subsolution to \eqref{eq:1B}, as stated in Proposition \ref{prop:approx}.

Fix  $\gamma\in \widetilde \cH(\nu_1; 0,t)$. By the definition of $ \widetilde \cH(\nu_1; 0,t)$, 
there exist $\{\eta_j\}_{j\in\N} \to 0$, $\{\nu_1^{\eta_j}\}_{j \in \N}\subset\cP(\T^n)$,  and $\{\gamma^{\eta_j}\}\subset \cR^+(\T^n\times\R^n\times[0,t])$
such that $\gamma^{\eta_j}\in \cH^{\eta_j}(\nu_1^{\eta_j}; 0,t)$ for all $j\in\N$, and
\begin{align*}
&\nu_1^{\eta_j}\rightharpoonup \nu_1 \quad\text{ weakly in the sense of measures in} \ 
\cR(\T^n), \\
&\gamma^{\eta_j}\rightharpoonup \gamma \quad\text{ weakly in the sense of measures in} \ 
\cR(\T^n\times\R^n\times[t_0,t_1]) \ \text{as} \ j\to\infty. 
\end{align*}   
Then, 
\begin{align*}
&\int_{\T^n \times \R^n\times[0,t]} 
\left( \tilde u_t(z,s)-\eta\Delta\tilde{u}(z,s)-\tr(A(z)D^2 \tilde u(z,s))+q\cdot D \tilde u(z,s) \right)
\,d\gam^{\eta_j}(z,q,s) \nonumber \\
=\, & 
\int_{\Tt^n}\tilde u(z,t)\,d\nu_1^{\eta_j} (z)
-\int_{\T^n}\tilde u(z,0)\,d\nu^{\gamma^{\eta_j}}(z).  
\end{align*}
Because of the definition of the Legendre transform in \eqref{legendre}, 
\[
q\cdot D\tilde u(z,s)\leq L(z,q)+H(z, D\tilde u(z,s)). 
\]
Accordingly, we have 
\begin{align*}
&
\int_{\T^n}\tilde{u}(z,t)\,d\nu_1^{\eta_j}-\int_{\T^n}\tilde{u}(z,0)\,d\nu^{\gamma^{\eta_j}}(z)\\
=\,&\, 
\int_{\T^n\times\R^n\times[0,t]} \left(\tilde{u}_t-\eta\Delta\tilde{u}-\tr(A(z)D^2 \tilde u)+q\cdot D\tilde{u} \right)\,d\gam^{\eta_j}(z,q,s)\\
\le\,&\,  
\int_{\T^n\times\R^n\times[0,t]}\left(\tilde{u}_t-\eta\Delta\tilde{u}-\tr(A(z)D^2 \tilde u)+H(z,D\tilde{u})+L(z,q) \right)\,d\gam^{\eta_j}(z,q,s)\\
\le\,&\,  
\int_{\T^n\times\R^n\times[0,t]}L(z,q)\,d\gam^{\eta_j}(z,q,s)
+\kappa(\alpha,\eta_j,\delta,\ep)t, 
\end{align*} 
where $\kappa(\alpha,\eta,\delta,\ep)$ is defined by \eqref{kappa} in Section \ref{subsec:approx} below. 
By taking a subsequence if necessary, we have 
\[
\nu^{\gamma^{\eta_j}}\rightharpoonup \nu^{\gamma}
\qquad\text{as} \ j\to\infty 
\]
for some $\nu^{\gamma}\in\cP(\T^n)$ weakly in the sense of measures.

Notice that if we send $\alpha\to 0$,  $j\to \infty$, and $\ep \to 0$ in this order, 
then we have 
\[
\kappa(\alpha,\eta_j,\delta,\ep)\to 0. 
\]
Thus, we obtain
\[
\int_{\Tt^n}u(x,t)\,d\nu_1 \leq \int_{\T^n \times \R^n\times[0,t]}  L(z,q)\,d\gam(z,q,s)+
\int_{\T^n}u(z,0)\,d\nu^\gamma(z).
\]
Because $\gamma\in \widetilde \cH(\nu_1; 0,t)$ is arbitrary, the statement follows. 
\end{proof}

To prove the opposite bound, we
 regularized \eqref{eq:1B} as in Problem \ref{P2}.
%\begin{problem}
%	\label{P2B} 
That is, for $\ep>0$,  
we	find $u^\ep:\Tt^n\times [0,\infty)\to \Rr$ solving
	\begin{equation} \label{eq:apB}
	\begin{cases}
	u^\ep_t -\tr(A(x)D^2 u^\ep)+ H(x,Du^\ep) = \ep \Del u^\ep \quad &\text { in } \T^n \times (0,\infty),\\
	u(x,0) = u_0(x) \quad &\text{ on } \T^n.
	\end{cases}
	\end{equation}
%\end{problem}

If Assumptions \ref{A1}-\ref{A3} hold, \eqref{eq:apB} has a unique solution, $u^\ep\in C^2(\Tt^n\times [0,+\infty))$.  
Moreover, $u^\ep$ is Lipschitz continuous uniformly in $\ep \in (0,1)$.
Further, by the standard viscosity solution theory, $u^\ep\to u$ locally uniformly, as $\ep\to 0$, where $u$ solves \eqref{eq:1B}.

Now, we use  the nonlinear adjoint method, see
\cite{E3, T1}, to construct measures that satisfy 
an approximated holonomy condition.

\begin{lem}\label{L1B}
For $\ep>0$, let $u^\ep$ solve \eqref{eq:apB}. 
For any $\nu_1\in \cP(\Tt^n)$, there exist a probability measure $\nu_0^{\ep}\in \cP(\Tt^n)$
and $\gamma^\ep\in\cH^{\ep}(\nu_0^{\ep},\nu_1;t_0,t_1)$. 
Moreover, $q=D_pH(x, Du^\ep(x,t))$  $\gamma^\ep$-almost everywhere. 
\end{lem}
\begin{proof}
For $\varphi\in C^2(\Tt^n\times [t_0,t_1])$, the linearization of \eqref{eq:apB} around the solution, $u^\ep$, is
	\[
	\cL^\ep[\varphi]= 
	\varphi_t -a^{ij}\varphi_{x_ix_j}+ D_pH(x,Du^\ep)\cdot D\varphi - \ep\Del \varphi, 
	\]
	where we use Einstein's summation convention. Accordingly, 
	the corresponding adjoint equation is the Fokker-Planck equation
	\begin{equation} \label{eq:adjoint-Appen}
	\begin{cases}
	-\sig^\ep_t -(a^{ij}\sig^\ep)_{x_ix_j} - \text{div}(D_pH(x,Du^{\ep})\sig^{\ep}) = \ep \Del \sig^\ep \quad &\text { in } \T^n \times (t_0,t_1),\\
	\sig^\ep(x,t_1) = \nu_1 \quad &\text{ on } \T^n.    
	\end{cases}
	\end{equation}
	By standard properties of the Fokker-Planck equation,
	\[
	\sig^{\ep}>0 \ \text{on} \ \T^n\times[t_0,t_1), \quad \text{ and } \quad 
	\int_{\T^n} \sig^{\ep}(x,t)\,dx=1 \ \text{for all} \ t\in[t_0,t_1).
	\]
	Next, 
	for each $\ep>0$ and $t\in[t_0,t_1]$, 
	let $\beta^{\ep}_t \in \mathcal{P}(\T^n \times \R^n)$ 
	be the probability measure determined by
	\begin{equation*}%\label{def-nu-ep-second}
	\int_{\T^n} \psi(x,Du^{\ep}) \sig^{\ep}(x,t)\,dx
	=\int_{\T^n \times \R^n} \psi(x,p)\,d\beta^{\ep}_t(x,p)
	\end{equation*}
	for all $\psi \in C_c(\T^n \times \R^n)$.
	For $t\in[t_0,t_1]$, let $\gam^{\ep}_t \in \mathcal{P}(\T^n \times \R^n)$ the pullback of $\beta^{\ep}_t$ by the map $\Phi(x,q)=(x,D_q L(x,q))$; that is,  
	\begin{equation*}%\label{def-mu-ep-second}
	\int_{\T^n \times \R^n} \psi(x,p)\,d\beta^{\ep}_t(x,p)=\int_{\T^n \times \R^n} \psi(x,D_q L(x,q))\,d\gam^{\ep}_t(x,q)
	\end{equation*}
	for all $\psi \in C_c(\T^n \times \R^n)$.

	Define the measures $\beta^\ep, \gam^\ep \in \cR(\T^n \times \R^n\times[t_0,t_1])$
	by  
	\begin{align*}
	&
	\int_{\T^n\times\R^n\times[t_0,t_1]}f\,d\beta^\ep
	=
	\int_{t_0}^{t_1}\int_{\T^n\times\R^n}f(\cdot,t)\,d\beta^\ep_t\,dt, \\ 
	& 
	\int_{\T^n\times\R^n\times[t_0,t_1]}f\,d\gam^\ep
	=
	\int_{t_0}^{t_1}\int_{\T^n\times\R^n}f(\cdot,t)\,d\gam^\ep_t\,dt 
	\end{align*}
	for any $f\in C_c(\T^n\times\R^n\times[t_0,t_1])$.
	
	Multiplying the first equation in \eqref{eq:adjoint-Appen} by  an arbitrary function, $\varphi\in C^2(\T^n\times[t_0,t_1])$, and integrating on $\T^n$, we gather
	\begin{align*}
	\ep\int_{\T^n}\sig^\ep\Del\varphi\,dx
	=&\, 
	-\int_{\T^n}\varphi\sig^\ep_t\,dx
	+\int_{\T^n}(-a^{ij}(x)\varphi_{x_ix_j}+D_pH(x,Du^\ep)\cdot D\varphi)\sig^\ep\,dx\\ 
	=&\, 
	-\int_{\T^n}(\varphi\sig^\ep)_t\,dx
	+\int_{\T^n}(\varphi_t-a^{ij}(x)\varphi_{x_ix_j}+D_pH(x,Du^\ep)\cdot D\varphi) \sig^\ep\,dx. 
	\end{align*}
	Next, integrating on $[t_0,t_1]$, we deduce the identity
	\begin{align*}
	&\ep\int_{t_0}^{t_1}\int_{\T^n}\sig^\ep\Del\varphi\,dxdt
	=\ep\int_{\T^n\times\R^n\times[t_0,t_1]}\Del\varphi\,d\gamma^{\ep}(x,q,t) 
	\\
	=&\, 
	-\int_{t_0}^{t_1}\int_{\T^n}(\varphi\sig^\ep)_t\,dxdt
	+\int_{t_0}^{t_1}\int_{\T^n}(\varphi_t-a^{ij}(x)\varphi_{x_ix_j}+D_pH(x,Du^\ep)\cdot D\varphi) \sig^\ep\,dxdt\\
	=&\, 
	-\left[\int_{\T^n}\varphi(\cdot,t_1)\sig^\ep(\cdot,t_1)\,dx-\int_{\T^n}\varphi(\cdot,t_0)\sig^\ep(\cdot,t_0)\,dx\right]
	\\
	&
	\quad \quad\quad\quad\quad \quad\quad\quad+\int_{t_0}^{t_1}\int_{\T^n\times\R^n}(\varphi_t-a^{ij}(x)\varphi_{x_ix_j}+q\cdot D\varphi) \,d\gam^\ep_t(x,q)\\
	=&\, 
	-\left[\int_{\T^n}\varphi(\cdot,t_1)\,d\nu_1-\int_{\T^n}\varphi(\cdot,t_0)\,d\nu_0^\ep\right]\\
	&
	\quad \quad\quad\quad \quad \quad\quad\quad+\int_{\T^n\times\R^n\times[t_0,t_1]}(\varphi_t-a^{ij}(x)\varphi_{x_ix_j}+q\cdot D\varphi)\, d\gam^\ep(x,q,t),  
	\end{align*}
	where $d\nu_0^{\ep}:=\sigma^\ep(x,t_0)\,dx$, which implies 
$\gam^\ep\in\cH^\ep(\nu_0^{\ep}, \nu_1;t_0,t_1)$. 
\end{proof}

\begin{cor}\label{cor:non-emptyB}
Under Assumptions \ref{A1}--\ref{A3}, for all $0<t_0<t_1$ and $\nu_1\in\cP(\T^n)$, 
\[
\widetilde \cH(\nu_1;t_0,t_1)\neq \emptyset. 
\]
\end{cor}
\begin{proof}
Let $\nu_0^{\ep}\in \cP(\Tt^n)$ and $\gamma^\ep\in\cH^\ep(\nu_0^{\ep},\nu_1;t_0,t_1)$ be measures 
given by Lemma \ref{L1B}. 

Because
\[
\|Du^\ep(\cdot,t)\|_{\Li(\T^n)}\le C\quad \text{for some} \ C>0, \ \text{for all} \ t\in[t_0,t_1], 
\]
there exists a sequence $\{\ep_j\}\to0$ such that 
\begin{equation}\label{subseq1B}
\nu_0^{\ep_j}\rightharpoonup \nu_0\in\cP(\T^n), \quad
\gam^{\ep_j}\rightharpoonup \gam\in\cR(\T^n\times\R^n\times[t_0,t_1])  
\quad\text{as} \ \ j\to\infty,  
\end{equation}
weakly in terms of measures on $\T^n$ and $\T^n\times\R^n\times[t_0,t_1]$, respectively. 
Thus, $\gam\in\widetilde\cH(\nu_0,\nu_1;t_0,t_1)$, which implies the conclusion.  
\end{proof}

Finally, we use Lemma \ref{L1B} to establish the opposite inequality to the one
in Lemma \ref{VLB}.

\begin{lem}\label{lem:ineq1B}
	For any $\nu\in\cP(\T^n)$ and $t>0$, we have
	\[
	\int_{\T^n}u(z,t)\,d\nu(z)
	\ge 
	\inf_{\gam\in \widetilde\cH(\nu;0,t)}
	\left\{
	\int_{\T^n\times\R^n\times[0,t]}L(z,q)\,d\gam(z,q,s)+\int_{\T^n}u_0(z)\,d\nu^\gam(z)
	\right\}. 
	\] 
\end{lem}

\begin{proof}
	For $s\in[0,t]$, let $\gam^\ep_s$ be the measure constructed in the proof of Lemma \ref{L1B}	for $t_0=0$ and $t_1=t$.
	By the properties of the Legendre  transform
	\[L(z,q)=D_pH(z,D_qL(z,q))\cdot D_qL(z,q)-H(z,D_qL(z,q)).
	\] 
	Therefore, 
	\begin{align*}
	&\int_{\T^n \times \R^n} L(z,q)\,d\gam^{\ep}_s(z,q)\\
	=&\,
	\int_{\T^n \times \R^n} 
	(D_pH(z,D_qL(z,q))\cdot D_qL(z,q)-H(z,D_qL(z,q)))\,d\gam^{\ep}_s(z,q)\\ 
	=&\,
	\int_{\T^n \times \R^n} 
	(D_pH(z,p)\cdot p -H(z,p))\,d\beta^{\ep}_s(z,p)\\
	=&\,
	\int_{\T^n} 
	(D_pH(x,Du^{\ep})\cdot Du^\ep -H(x,Du^{\ep}))\sig^{\ep}(x,s)\,dx 
	\end{align*}
	for all $s\in[0,t]$. 
	Moreover, integrating by parts and using the adjoint equation and \eqref{eq:apB}, we obtain
	\begin{align*}
	&
	\int_{\T^n} 
	(D_pH(x,Du^{\ep})\cdot Du^{\ep} -H(x,Du^{\ep}))\sig^{\ep}\,dx\\
	=&\,
	\int_{\T^n} 
	-\div(D_pH(x,Du^{\ep})\sig^{\ep})u^{\ep}-H(x,Du^{\ep})\sig^{\ep}\,dx\\
	=&\,
	\int_{\T^n} 
	(\sig^{\ep}_t+\ep\Del\sig^\ep+(a^{ij}\sig^\ep)_{x_ix_j})u^{\ep}
	+(u^\ep_t-\ep\Del u^\ep-a^{ij}u^\ep_{x_ix_j})\sig^{\ep}\,dx\\
	=&\,
	\int_{\T^n} 
	(u^\ep\sig^{\ep})_t\,dx. 
	\end{align*}
	
Integrating on $[0,t]$ yields 
\begin{align*}
\int_{\T^n \times \R^n\times[0,t]} L(x,q)\,d\gam^{\ep}(x,q,s)
&=\,
\int_0^t\int_{\T^n \times \R^n} L(x,q)\,d\gam^{\ep}_s(x,q)ds\\
&=\, 
\int_{\T^n}u^\ep(x,t)\,d\nu
-\int_{\T^n} u_0(x)\,d\nu^\ep_0.  
\end{align*}	
Taking subsequences $\{\gam^{\ep_j}\}$ and $\{\nu_0^{\ep_j}\}$ as in \eqref{subseq1B} yields  
\[
\int_{\T^n}u(z,t)\,d\nu(z)
=\int_{\T^n \times \R^n \times [0,t]} L(z,q)\,d\gam(z,q,s)+
\int_{\T^n} u_0(z)\,d\nu_0(z). 
\]
Because 
$\gam\in\widetilde \cH(\nu_0,\nu;0,t)$,
we obtain the inequality claimed in the statement. 
\end{proof}

\begin{proof}[Proof of Theorem \ref{thm:rep B}]
The statement follows directly by combining Lemma \ref{L1B} with Lemma \ref{lem:ineq1B}.
\end{proof}

\subsection{Approximation and proof of Proposition \ref{prop:erg}}\label{subsec:approx}

We  first construct an approximation of viscosity solutions of \eqref{eq:1B} by $C^2$-subsolutions of an approximate equation. 
We begin by recalling the definition of the sup and inf-convolutions and some of their basic properties.

Let $w:\Tt^n\times [0,T]\to \Rr$ be a continuous function. 
The sup-convolution, $w^\ep$,  
and inf-convolution, $w_\ep$, of $w$ with respect to $x$ for $\ep>0$ are defined by 
\begin{align*}
&w^\ep(x,t):=\sup_{y\in\T^n}\left[u(y,t)-\frac{|x-y|^2}{2\ep}\right], \quad 
w_\ep(x,t):=\inf_{y\in\T^n}\left[u(y,t)+\frac{|x-y|^2}{2\ep}\right]. 
\end{align*}

\begin{proposition}\label{prop:basic} 
Let $w\in\Lip(\T^n\times[0,\infty))$, and set $L:=\|Dw\|_{L^\infty(\T^n\times(0,\infty))}$.   
We have
\begin{align*}
&
w^\ep(x,t)=\max_{ |x-y|\le 2L\ep}\left\{w(y,t)-\frac{|x-y|^2}{2\ep}\right\}, \quad
w_\ep(x,t)=\min_{ |x-y|\le 2L\ep}\left\{w(y,t)+\frac{|x-y|^2}{2\ep}\right\}.
\end{align*}
Moreover, $\|Dw^\ep\|_{L^\infty(\T^n\times(0,\infty))}\le 2L$ and 
$\|Dw_\ep\|_{L^\infty(\T^n\times(0,\infty))}\le 2L$.  

\end{proposition}
\begin{proof}
We only give a proof for $w^{\ep}$. 
Take $x,y\in\T^n$ so that $2L\ep <|x-y|$. Since 
$w(y,t)-w(x,t)\le L|x-y|<\frac{|x-y|^2}{2\ep}$, 
we have 
$w(y,t)-\frac{|x-y|^2}{2\ep}<w(x,t)$, which implies the first claim.

To get the Lipschitz estimate for $w^\ep$, 
for a fixed $x\in\T^n$, take $z_x\in\ol{B}(x,2L\ep)$ so that 
$w^\ep(x,t)=w(z_x,t)-\frac{|x-z_x|^2}{2\ep}$. 
Then, 
\begin{align*}
w^\ep(x,t)-w^\ep(y,t)
&\le\, 
\left(w(z_x,t)-\frac{|x-z_x|^2}{2\ep}\right)
-\left(w(z_x,t)-\frac{|y-z_x|^2}{2\ep}\right) \\
&=\, 
\frac{(y-x)\cdot (x+y-2z_x)}{2\ep}
\le
\frac{|x-y|(|x-z_x|+|y-z_x|)}{2\ep}\\
&\le\, 
\frac{4L+C}{2}|x-y|
\end{align*}
if $|x-y|\le C\ep$ for any $C>0$,   
which implies the conclusion. 
\end{proof}

It is well known that 
the inf-sup convolution $(w^{\ep+\del})_\del$ for $\ep, \del>0$  gives a $C^{1,1}$ approximation of $w$ in $x$ (see \cite{LL} for instance). 
\begin{proposition}\label{prop:double-conv1}
We have $(w^{\ep+\del})_\del\in \Lip([0,\infty); C^{1,1}(\T^n))$. 
Moreover, for each $t>0$,
\begin{equation}\label{est:Hessian}
-\frac{1}{\ep} I\le D^2(w^{\ep+\del})_\del(x,t) \le \frac{1}{\del} I  
\quad\text{for} \ a.e. \ x\in\R^n.   
\end{equation}
\end{proposition}
\begin{proof}
It is clear that $(w^{\ep+\del})_\del$ is $(1/2\del)$-semiconcave. 
Because the inf and sup convolutions satisfy the semigroup property, that is,
$w^{\ep+\del}=(w^{\ep})^\del$, $w_{\ep+\del}=(w_{\ep})_\del$ for $\ep, \del>0$,  
we have 
\[
(w^{\ep+\del})_\del=((w^\ep)^\del)_\del, 
\] 
$(w^{\ep+\del})_\del$ is $(1/2\ep)$-semiconvex in light of \cite[Proposition 4.5]{CKSS}. 
Therefore, 
$(w^{\ep+\del})_\del\in \Lip([0,\infty); C^{1,1}(\T^n))$, and \eqref{est:Hessian} follows. 
\end{proof}

\begin{lem}\label{lem:eta}
Let $w$ be a Lipschitz subsolution to \eqref{eq:1B}.
Then, there exists a modulus of continuity
 $\omega\in C([0,\infty))$,  nondecreasing and with $\omega(0)=0$,
a function  $\del_0(\ep,\eta)$, defined for $\ep, \eta>0$, 
such that  
for $\ep, \eta>0$, $\del_0=\del_0(\ep,\eta)$
and any 
 $\del\in(0,\del_0]$, 
 and any $(x,t) \in \T^n \times (0,\infty)$
 where 
 $(w^{\ep+\del})_\del$ is twice differentiable in $x$ and  differentiable in $t$
 we have  at $(x,t)$ the following inequaliy
	\begin{equation*}%\label{eq:eta}
	((w^{\ep+\del})_\del)_t -\eta\Del (w^{\ep+\del})_\del-\tr(A(x)D^2(w^{\ep+\del})_\del)+H(x,D(w^{\ep+\del})_\del)\le \om(\ep)+\frac{n\eta}{\ep}.
	\end{equation*}

\end{lem}

\begin{proof}
We first notice that by standard  viscosity solution theory, $w^\ep$ satisfies 
\[
w^\ep_t-\tr(A(x)D^2w^\ep)+H(x,Dw^\ep)\le \om(\ep) 
\]
in the sense of viscosity solutions for some nondecreasing function $\omega\in C([0,\infty))$ with 
$\omega(0)=0$. 
Furthermore, because  $w^{\ep}+|x|^2/(2\ep)$ is convex, 
we see that 
$-\Delta w^{\ep}\le n/\ep$ in $\T^n\times(0,\infty)$ in the sense of viscosity solutions. 
Thus,
\[
w^\ep_t-\eta\Del w^\ep-\tr(A(x)D^2w^\ep)+H(x,Dw^\ep)\le \om(\ep)+\frac{n\eta}{\ep}   
\]
in the sense of viscosity solutions.

Let $\tilde w=(w^{\ep+\del})_\del$. Note that $\tilde w\geq w^\ep$ on $\T^n$. 
Now, let $(\hat{x},\hat{t})\in\T^n\times(0,\infty)$ be a point where $\tilde w$ is twice differentiable in $x$ and  differentiable in $t$. 
Select a function, $\varphi\in C^2(\T^n\times(0,\infty))$, such that 
	$\tilde w-\varphi$ has a maximum at $(\hat{x},\hat{t})$. 
At this point either $\tilde w(\hat{x},\hat{t})=w^{\ep}(\hat{x},\hat{t})$ or
$\tilde w(\hat{x},\hat{t})>w^{\ep}(\hat{x},\hat{t})$.
In the first alternative, that is, when 
$\tilde w(\hat{x},\hat{t})=w^{\ep}(\hat{x},\hat{t})$, 
 $w^\ep-\varphi$ has a maximum at $(\hat{x}, \hat{t})$.
	Thus, 
\begin{align*}
&\tilde{w}_t(\hat{x},\hat{t})-\eta\Del \tilde{w}(\hat{x},\hat{t})-\tr(A(\hat{x})D^2\tilde{w}(\hat{x},\hat{t}))+H(x,D\tilde{w}(\hat{x},\hat{t}))\\
=&\, 
\varphi_t(\hat{x},\hat{t})-\eta\Del \varphi(\hat{x},\hat{t})-\tr(A(\hat{x})D^2\varphi(\hat{x},\hat{t}))+H(x,D\varphi(\hat{x},\hat{t}))
\le \om(\ep)+\frac{n\eta}{\ep}. 
\end{align*}

In the second alternative, that is, if $\tilde w(\hat{x},\hat{t})>w^\ep(\hat{x},\hat{t})$, 
	by \cite[Proposition 4.4]{CKSS}, 
	$1/\del$ is one of the eigenvalues of $D^2\tilde{w}(\hat{x},\hat{t})$.  
	Moreover, by using \eqref{est:Hessian}, we get 
	\[
	\eta\Del \tilde{w}(\hat{x},\hat{t})\ge \eta\left(\frac{1}{\del}-\frac{n-1}{\ep}\right). 
	\]
	Letting $\Lam:=\max_{x\in\T^n}\eig(A(x))\ge0$, 
	where $\eig(A(x))$ are the eigenvalues of $A(x)$, we similarly have 
	\[
	\tr(A(\hat{x})D^2\tilde{w}(\hat{x},\hat{t}))\ge -\frac{(n-1)\Lam}{\ep}. 
	\]
Set $L:=\|w_t\|_{\Li(\T^n)}+\|Dw\|_{\Li(\T^n)}$, and 
$\bar C=\max_{|p|\leq 2L}H(x,p)$.  
Moreover, we set
	\begin{equation}
	\label{mdelta0}
 \del_0(\ep,\eta)=\frac{\ep\eta}{(n-1)(\Lam+\eta)+(L+\bar C)\ep}.  
	\end{equation}

Combining the preceding estimates, for all $0<\del<\del_0=\del_0(\ep,\eta)$, 
we get 
\begin{align*}
&
\tilde{w}_t(\hat{x},\hat{t})-\eta\Del \tilde{w}(\hat{x},\hat{t})-\tr(A(\hat{x})D^2\tilde{w}(\hat{x},\hat{t}))+H(x,D\tilde{w}(\hat{x},\hat{t}))\\
\le&\, 
L-\eta\left(\frac{1}{\del}-\frac{n-1}{\ep}\right)+\frac{(n-1)\Lam}{\ep}
+\bar C
= 
-\frac{\eta}{\del}+\frac{(n-1)(\Lam+\eta)+(L+\bar C)\ep}{\ep} \leq 0, 
\end{align*}
by the choice of $\delta_0$ in \eqref{mdelta0}, 
which finishes the proof. 
\end{proof}

\medskip
Next, we regularize $(w^{\ep+\del})_\del$ further by using standard mollifiers to obtain a
$C^2$ subsolution.   
Let $\theta \in C_c^\infty(\R^{n},[0,\infty))$ and $\rho \in C_c^\infty(\R,[0,\infty))$
be symmetric standard mollifiers; that is, 
$\supp \theta\subset\ol{B}(0,1)\subset\R^{n}$, $\supp \rho\subset\ol{B}(0,1)\subset\R$, $\theta(x)=\theta(-x)$, $\rho(s)=\rho(-s)$,
and $\|\theta\|_{L^{1}(\R^{n})}=\|\rho\|_{L^{1}(\R)}=1$. 
For each $\alpha>0$, set $\theta^\alpha(x):=\alpha^{-n} \theta(\alpha^{-1}x)$ for $x\in \R^{n}$, 
and $\rho^{\alpha}(t):=\alpha^{-1} \rho(\alpha^{-1}t)$ for $t\in \R$. 
We define the function $w^{\alpha,\ep,\del}\in C^\infty(\T^n\times[\al,\infty))$ by 
\begin{equation}\label{func:al-ep-del}
w^{\al,\ep,\del}(x,t)
:=
\int_{0}^{\infty} \rho^\al(s)\int_{\T^n} \theta^\al(y)(w^{\ep+\del})_\del(x-y,t-s)\,dy ds
\end{equation}
for all $(x,t)\in\T^n\times[\al,\infty)$. 

\begin{lem}\label{lem:est-R1}
Let $w\in\Lip(\T^n\times(0,\infty))$ and $w^{\al,\ep,\del}$ be given by \eqref{func:al-ep-del}. Then,
there exists a
constant  $C>0$,  independent of $\al, \ep, \del$, such that 
for all $(x,t)\in\T^n\times[\al,\infty)$, we have
\begin{enumerate}
\item[(i)] 
\begin{align*}
&\Big|
\tr(A(x)D^2w^{\al,\ep,\del}(x,t))\\
&\quad 
-\int_0^\infty\int_{\T^n}\rho^\al(s)\theta^\al(y)\tr\left(A(x-y)D^2(w^{\ep+\del})_\del(x-y,t-s)\right) \,dyds
\Big|\\
\le\,  &
C\al\max\left\{\frac{1}{\ep}, \frac{1}{\del}\right\}, 
\end{align*}
\item[(ii)]  
\begin{align*}
	&H(x,Dw^{\al,\ep,\del}(x,t))\\
	&\quad -\int_0^\infty\int_{\T^n} \rho^\al(s)\theta^\al(y)H(x-y,D(w^{\ep+\del})_\del(x-y,t-s))\,dyds\le 
C\al.
\end{align*}
\end{enumerate}

\end{lem}
\begin{proof}
	We denote, respectively, $w^{\al,\ep,\del}$ and $(w^{\ep+\del})_\del$ by $w^\al$ and $w$ for simplicity in the proof. 
	We begin by proving the first inequality. 
	We have 
	\begin{align*}
	&\tr(A(x)D^2w^\al(x,t))
	=\, 
	\tr\left(A(x)\int_0^\infty\int_{\T^n}\rho^\al(s)\theta^\al(y)D^2w(x-y,t-s)\,dyds\right)\\
	=&\, 
	\int_0^\infty\int_{\T^n}\rho^\al(s)\theta^\al(y)\tr\left(A(x-y)D^2w(x-y)\right)\,dyds\\
	&+\int_0^\infty\int_{|y|\le\al}\rho^\al(s)\theta^\al(y)\tr\left((A(x)-A(x-y))D^2w(x-y)\right)\,dyds. 
	\end{align*}
	By Proposition \ref{prop:double-conv1}, we have, for each $t>0$,
	\[
	|D^2w(x,t)|\le C\max\left\{\frac{1}{\ep}, \frac{1}{\del}\right\} 
	\quad\text{for} \ a.e. \ x\in\T^n, 
	\]
	which implies (i). 
	
	By the convexity of $H$, Jensen's inequality
	implies
	\begin{align*}
	H(x,Dw^{\al}(x,t))
	&=\, 
	H\left(x,\int_0^\infty\int_{\T^n} \rho^\al(s)\theta^\al(y)Dw(x-y,t-s)\,dyds\right) 
	\\
	&\,\leq 
	\int_0^\infty\int_{\T^n} \rho^\al(s)\theta^\al(y) H(x,Dw(x-y,t-s))\,dyds. 
	\end{align*} 
	
	By Proposition \ref{prop:basic} and (A3),
	\[
	|H(x-y,Dw(x-y,t-s))-H(x,Dw(x-y,t-s))|\le C\al \quad 
	\]
	for all $y\in B(x,\al)$ and $t-s>0$.  
	Thus, we finish the proof. 
\end{proof}

\begin{pro}\label{prop:approx}
Let $w$ be a Lipschitz subsolution to \eqref{eq:1B}.
For $\al, \ep, \del>0$, 
let $w^{\al,\ep,\del}\in C^2(\T^n\times[\al,\infty))$ be the function 
defined by \eqref{func:al-ep-del}. 
For $\eta>0$, let $\del_0=\del_0(\ep,\eta)$ be the constant given by \eqref{mdelta0}. 
Then, we have 
\begin{equation}\label{eq:eta}
(w^{\al,\ep,\del})_t-\eta\Del w^{\al,\ep,\del}
-\tr(A(x)D^2w^{\al,\ep,\del})+H(x,Dw^{\al,\ep,\del})
\le 
\kappa(\alpha,\eta, \delta, \ep)
\end{equation}
for all $\del\in(0,\del_0]$ and $(x,t)\in\T^n\times[\al,\infty)$. 
Here,  
\begin{equation}
\label{kappa}
\kappa(\alpha,\eta, \delta, \ep)
:=\om(\ep)+\frac{n\eta}{\ep}+C\max\left\{\frac{1}{\ep},\frac{1}{\del}\right\}\al 
\end{equation}
where 
$\omega\in C([0,\infty))$ is a nondecreasing function $\omega\in C([0,\infty))$ with 
$\omega(0)=0$, 
and $C$ is a positive constant. 
\end{pro}
\begin{proof}
Proposition {\rm\ref{prop:approx}} is a straightforward result of Lemmas \ref{lem:eta}, \ref{lem:est-R1}.   
\end{proof}

%\begin{remark}
%It is worth emphasizing that $C^2$ regularity is {\color{red} very important} to be a test function for 
%generalized holonomic measures. 
%\end{remark}

We finally apply this approximation procedure to give the characterization 
of the ergodic constant for \eqref{eeB} in terms of generalized Mather measures stated in Proposition \ref{prop:erg}

\begin{proof}[Proof of Proposition \ref{prop:erg}]
Take $\mu\in \widetilde \cH$. 
By the definition of $  \widetilde\cH$, there exist $\{\eta_j\}_{j\in \N}\to0$ and $\{\mu^{\eta_j}\}_{j\in\N}\subset\cP(\T^n \times \R^n)$ such that $\mu^{\eta_j}\in\cH^{\eta_j}$, and 
\[
\mu^{\eta_j}\rightharpoonup \mu \quad\text{ weakly in the sense of measures in} \ 
\cP(\T^n \times \R^n). 
\]

Let $v$ be a Lipschitz viscosity solution to \eqref{eeB}.
For $\alpha,\ep,\delta>0$, denote by 
\begin{equation*}
v^{\al,\ep,\del}(x)
:=
\int_{\T^n} \theta^\al(y)(v^{\ep+\del})_\del(x-y)\,dy.
\end{equation*}
Then, $v^{\al,\ep,\del}\in C^2(\T^n)$.
Let $\tilde{v}:=v^{\al,\ep,\del}$, to simplify the notation.
Because of Proposition \ref{prop:approx}, we have 
\begin{align*}
\kappa(\alpha,\eta,\delta,\ep)
&\,\ge
\int_{\T^n\times\R^n}
\left(-\eta_j\Delta\tilde{v}-\tr(A(z)D^2\tilde{v})+H(z,D\tilde{v}) \right)\, d\mu^{\eta_j}(z,q)\\
&\,\ge
\int_{\T^n\times\R^n}
\left(-\eta_j\Delta\tilde{v}-\tr(A(z)D^2\tilde{v})+q\cdot D\tilde{v}-L(z,q) \right)\, d\mu^{\eta_j}(z,q)\\
&\,=
-\int_{\T^n\times\R^n}L(z,q)\, d\mu^{\eta_j}(z,q). 
\end{align*}
We let $\alpha\to 0$,  $j\to \infty$, and $\ep \to 0$, in this order, 
to get 
\[
\int_{\T^n\times\R^n}L(z,q)\, d\mu(z,q)\ge 0 
\quad \text{ for all} \ \mu\in  \widetilde\cH.  
\]

For $\ep>0$, let $(v^\ep,c^\ep)\in C^2(\T^n)\times\R$ solve
\begin{equation}\label{eq:v-ep}
-\ep\Delta v^\ep-\tr(A(x)D^2v^\ep)+H(x,Dv^\ep)=c^\ep\quad\text{ in} \ \T^n. 
\end{equation}
The constant $c^\ep$ is unique.
Besides, thanks to Assumptions \ref{A1}-\ref{A3}, we have that $v^\ep$ is Lipschitz continuous uniformly in $\ep \in (0,1)$.
Because the ergodic constant $c$ for \eqref{ee} was normalized to be $0$, 
we see that $c^\ep\to0$ as $\ep\to0$. 
Let $\theta^\ep$ solve the associated adjoint equation
\[
\left\{
\begin{array}{ll}
&-\ep\Delta\theta^{\ep}-(a^{ij}\theta^\ep)_{x_ix_j}-\text{div}(D_pH(x,Dv^\ep)\theta^\ep)=0 \quad\text{ in} \ \T^n, \\
&
\displaystyle
\int_{\T^n}\theta^\ep(x)\,dx=1.
\end{array}
\right. 
\]
Next, we define a measure,  $\mu^\ep\in\cP(\T^n\times\R^n)$, as
follows
	\begin{equation*}%\label{def-nu-ep-second}
	\int_{\T^n} \psi(x,Dv^{\ep}) \theta^{\ep}(x)\,dx
	=\int_{\T^n \times \R^n} \psi(x,D_q L(x,q))\,d\mu^{\ep}(x,q)
	\end{equation*}
for all $\psi \in C_c(\T^n\times \R^n)$.% as in the proof of Lemma \ref{L1B}. 

Multiplying \eqref{eq:v-ep} by $\theta^\ep$, integrating on $\T^n$, and using the integration by parts, we get 
\begin{align*}
c^\ep&=\, 
\int_{\T^n}
\Big(-\ep\Delta v^\ep-\tr(A(x)D^2v^\ep)+H(x,Dv^\ep)\Big)\theta^\ep\,dx\\
&=\,
\int_{\T^n}
\Big(H(x,Dv^\ep)-D_pH(x,Dv^\ep)\cdot Dv^\ep\Big)\theta^\ep\,dx
=\int_{\T^n\times\R^n}L(x,q)\,d\mu^\ep. 
\end{align*}
By a similar argument to the one in Lemma \ref{L1}, we see that $\mu^\ep\in\cH^\ep$. 
By extracting a subsequence if necessary, there exists a sequence $\{\ep_j\} \to 0$ such that
\[
\mu^{\ep_j}\rightharpoonup \mu
\quad
\text{ weakly in the sense of measures in} \ \cP(\T^n\times\R^n)
\]
as $j\to\infty$ for some $\mu\in  \widetilde\cH$, and
\[
\int_{\T^n\times\R^n}L(x,q)\,d\mu=0. 
\] 
This also shows that $\mu\in\widetilde{\cM}$, and finishes the proof. 
\end{proof}

%%%%%%%%%%%%%%%%%%%%%%%%%%%%%%%%%%%%%%%%%%%%%%%%%
%%%%%%%%%%%%%%%%%%%%%%%%%%%%%%%%%%%%%%%%%%%%%%%%%
%Reference

\bibliographystyle{plain}
% Diogo
\IfFileExists{"/Users/gomesd/mfgDGOFFICE.bib"}{\bibliography{/Users/gomesd/mfgDGOFFICE.bib}}{\bibliography{mfg.bib}}

\end{document}